\documentclass[12pt]{extarticle}
\usepackage{amsmath, amsthm, amssymb, xcolor}

\usepackage[colorlinks=true,linkcolor=blue,urlcolor=blue]{hyperref}
\usepackage{graphicx}
\usepackage{caption}
\usepackage{mathtools}
\usepackage{enumerate}
\usepackage{verbatim}
\usepackage{booktabs}
\usepackage{tikz,tikz-cd,tikz-3dplot}
\usepackage{amssymb}
\usetikzlibrary{matrix}
\usetikzlibrary{arrows}
\usetikzlibrary{positioning}
\usepackage{algorithm}
\usepackage{nicematrix}
\usepackage[noend]{algpseudocode}
\usepackage{caption}
\usepackage[normalem]{ulem}
\usepackage{subcaption}
\oddsidemargin \evensidemargin
\usepackage{makecell}
\usepackage{array}

\definecolor{wine}{HTML}{CE2F23}
\definecolor{salmon}{HTML}{EB8677}
\definecolor{mag}{HTML}{E087E8}
\definecolor{gre}{HTML}{74B638}
\definecolor{yel}{HTML}{BFBF41}

\newtheorem{theorem}{Theorem}
\newtheorem{proposition}[theorem]{Proposition}
\newtheorem{lemma}[theorem]{Lemma}
\newtheorem{corollary}[theorem]{Corollary}

\theoremstyle{definition}

\newtheorem{defprop}[theorem]{Definition/Proposition}

\newtheorem{remark}[theorem]{Remark}

\newtheorem{example}[theorem]{Example}

\numberwithin{theorem}{section}

\usepackage{thm-restate}
\newcommand{\PP}{\mathbb{P}}
\newcommand{\RR}{\mathbb{R}}

\newcommand{\CC}{\mathbb{C}}

\title{\bf Euclidean Distance Optimization\\ Within the Grassmannian}

\author{Hannah Friedman, Serkan Ho\c{s}ten and Andrea Rosana}

\date{}
\begin{document}
\maketitle

\begin{abstract}
\noindent 
Given a subvariety of the Grassmannian and a data point, we seek to find a point on the subvariety minimizing the Euclidean distance to the data point. 
The number of complex critical points of this optimization problem is the Euclidean distance (ED) degree. We show that the ED discriminant of the whole Grassmannian, that is, the set of data points with a number of critical points different than the ED degree, is the discriminant of the characteristic polynomial of the data as a projection matrix. Another closely connected algebraic complexity measure for a subvariety in the Grassmannian is the Grassmann distance (GD) degree, which is the number of complex critical points of the distance optimization problem when the data point itself is in the Grassmannian.  
We give formulae for ED and GD degrees of geometrically meaningful subvarieties of the Grassmannian, namely, products of smaller Grassmannians, matroid realization varieties, and Schubert varieties. 
\end{abstract}

\section{Introduction}

Our main object of study is the Grassmannian, realized as the affine variety of rank $k$ orthogonal projection matrices: 
\begin{equation}\label{eq:pGr}
 \mathrm{pGr}(k,n) = \{P \in \mathrm{Sym}^2\CC^n \,:\, P^2 = P \mbox{ and } \mathrm{trace}(P) =k\}. 
 \end{equation}
Following \cite{DFRS}, we call this the \emph{projection Grassmannian} to distinguish it from the Pl\"ucker embedding of the Grassmannian $\mathrm{Gr}(k,n) \subseteq \PP^{\binom{n}{k}-1}$; see Section~\ref{sec:EDopt}. 
Let $\mathcal{M}$ be an irreducible subvariety  of $\mathrm{pGr}(k,n)$ that is defined over $\RR$ and has Zariski dense real points.
We call $\mathcal{M}$ a \emph{model}, inspired by 
algebraic statistics. 
Given a data point $Q \in \mathrm{Sym}^2 \mathbb R^n$, we consider the problem of finding the closest point to $Q$ on $\mathcal M$:
\begin{equation} \label{eq:Euclidean-distance-opt}
 \textrm{minimize} \quad || P - Q||_2^2 \quad \textrm{subject to} \quad P \in \mathcal{M}(\RR) \subset \mathrm{pGr}(k,n).
\end{equation}
Via $||P-Q||_2^2 = \mathrm{trace}(P^2) - 2 \mathrm{trace}(PQ) + \mathrm{trace}(Q^2)$, the Euclidean distance optimization problem \eqref{eq:Euclidean-distance-opt} is equivalent to
\begin{equation} \label{eq:trace-opt}
 \textrm{maximize} \quad \mathrm{trace}(PQ) \quad \textrm{subject to} \quad P \in \mathcal{M}(\RR) \subset \mathrm{pGr}(k,n).
\end{equation}
The \emph{Euclidean distance (ED) degree} \cite{DHOST} of $\mathcal M$ is the number of complex critical points of \eqref{eq:trace-opt} (or, equivalently, of \eqref{eq:Euclidean-distance-opt})  for generic data $Q \in \mathrm{Sym}^2\mathbb R^n$.

When $\mathcal M = \mathrm{pGr}(k,n)$, the problem \eqref{eq:trace-opt} becomes the multieigenvector problem \cite{FH}, and for generic $Q$, its critical points are the $\binom{n}{k}$ $Q$-invariant subspaces of $\mathbb R^n$.
Our first main result makes precise the word ``generic," thereby describing the ED discriminant \cite[Section 7]{DHOST} of the projection Grassmannian. 
\begin{theorem}\label{thm:disc1}
    For $\mathcal M = \mathrm{pGr}(k,n)$, the optimization problem \eqref{eq:trace-opt} has $\binom n k$ critical points if and only if all eigenvalues of $Q$ have multiplicity one.
\end{theorem}

Our main focus is to study the problem \eqref{eq:trace-opt} for data which lies not only in the ED discriminant of $\mathrm{pGr}(k,n)$, but in the projection Grassmannian itself.
The \emph{Grassmann distance (GD) degree} (Definition~\ref{def:GD}) is the number of complex critical points of \eqref{eq:trace-opt} for generic $Q~\in~\mathrm{pGr}(k,n)$.
The difference between the ED and GD degrees is central to our work.

The GD degree was first introduced in \cite{FRS} for $k = 2$ as an algebraic counterpart to the Grassmann distance complexity \cite{LR}. It was left open whether the GD degree is finite \cite[Remark 4.4]{FRS}. We prove the finiteness for $k=2$. 
\begin{restatable}{theorem}{gdfinite} \label{thm:gd-finite}
    Every model in $\mathrm{pGr}(2,n)$ has finite GD degree.
\end{restatable}

In this work, we deepen the study of the ED and GD degrees of natural classes of subvarieties of the Grassmannian.
We now give a brief overview of the main results for the three classes of models $\mathcal M$ we study.
Our first example is the product of two Grassmannians. 
For instance the SVD of a rectangular matrix yields a point in a product of two Grassmannians; see the discussion on correspondence analysis in \cite[Section 5.2]{FH}.
\begin{theorem}\label{thm:prod_Grassmannian}
    Suppose $k,r,s,j,\ell \in [n]$ are such that $k \leq n/2, j + \ell = k$, and $0 \leq s \leq r \leq r+s \leq n$.
    The ED degree of $\mathrm{pGr}(j, r) \times \mathrm{pGr}(\ell, s) \subseteq \mathrm{pGr}(k,n)$ is $\binom{r}{j} \cdot \binom{s}{\ell}$. 
    The GD degree is
    \begin{align*}
        \begin{cases}
            \binom{\min\{s,k\}}
{\ell} &\textrm{ if } n = r + s\\
            \binom{\min\{r,k\}}{j} \cdot \binom{\min\{s,k\}}
{\ell} &\textrm { elseif } r \leq n - k\\
            \binom{n-r}{\ell} \cdot \binom{s}{\ell} &\textrm { elseif } r > n - k.
        \end{cases}
    \end{align*}
    Furthermore, if the data point is real, all critical points are real.
\end{theorem}   
Matroid varieties are a combinatorially rich class of Grassmann subvarieties.
Given a realizable matroid, the realization space is the set of points in a Grassmannian realizing that matroid. 
The corresponding matroid variety is the Zariski closure of this set.
Up to isomorphism, a rank 2 matroid is uniquely determined by its number $\ell \leq n$ of loops (dependent elements)  and the sizes $\pi = (\pi_1, \ldots, \pi_r) \vdash n-\ell$ of its parallel classes (maximal sets such that every $2$-subset is dependent). 
\begin{theorem} \label{thm:ed/gd-rank 2-positroids} 
Let $M_{\pi,\ell}$ be a matroid of rank $2$ with $\ell$ loops whose parallel classes have sizes $\pi = (\pi_1, \ldots, \pi_r) \vdash n-\ell$, and 
let $\mathcal M_{\pi, \ell} \subset \mathrm{pGr}(2,n)$ be the corresponding matroid variety.  
\begin{enumerate}
    \item \emph{(Corollary~\ref{cor:matroidED})} If $r =2$, the ED degree of $\mathcal M_{\pi,\ell}$ is $\pi_1 \pi_2$. The GD degree of $\mathcal M_{\pi, 0}$ is $2$ if $2 \leq \pi_1 \leq n-2$ and $1$ otherwise. If $\ell>0$, the GD degree of $\mathcal M_{\pi,\ell}$  is $2^{\#\{i \in [2] \, : \, |\pi_i|>1\}}$.
    \item \emph{(Corollary~\ref{cor:GD-matroid})} If $r\geq 3$, the GD degree of $\mathcal{M}_{\pi, \ell}$ is equal to $2^{\#\{i \in [r]\, : \, |\pi_i| > 1\}}.$
\end{enumerate}
Furthermore, if the data point is real, all critical points are real.
\end{theorem}
\noindent Matroids of rank $k \geq 3$ are more complicated and their matroid varieties are not necessarily irreducible by Mn\"ev's Universality Theorem; see, e.g., \cite{CL}. 
However, in $\mathrm{pGr}(3,6)$, there are $38$ matroid varieties up to isomorphism, all of which are irreducible \cite{dukes}.  
We present a computational study of their ED and GD degrees in Table~\ref{tab:36-matroids}.

Finally, we will turn to Schubert varieties and extend the results of \cite[Section 6]{FRS} to Schubert varieties in ${\rm pGr}(k,n)$. 
The Schubert variety $\mathcal S_{i_1\cdots i_k}$ is the Zariski closure of the points in $\mathrm{pGr}(k,n)$ whose reduced row echelon forms have pivot columns indexed by $i_1, \ldots, i_k$. 
We say these pivot columns are cyclically adjacent if $i_{j+1} = (i_j\mod n) + 1$ for $j \in [k-1]$.
\begin{restatable}{theorem}{gdschubert} \label{thm:gd-schubert} The Schubert varieties $\mathcal S_{i_1\cdots i_k}$
in $\mathrm{pGr}(k,n)$ with cyclically adjacent pivot columns have GD degree one.
Moreover, this unique critical point is a rational function of the data point.
\end{restatable}
In particular, given a data point in Pl\"ucker coordinates, the unique critical point is found by setting certain Pl\"ucker coordinates to zero; see Section~\ref{sec:schubert} for details.

This article is organized as follows. 
In Section~\ref{sec:EDopt}, we discuss in detail the optimization problem \eqref{eq:trace-opt} for subvarieties of the Grassmannian and explain how to derive equations to solve this problem using different coordinates. We define ED-general subvarieties, namely those subvarieties whose ED critical points can be computed via saturating with respect to their singular locus only. We show that if a subvariety has dimension that is too big, it cannot be ED-general (Theorem \ref{thm:codim-sing-locus}).

In Section~\ref{sec:ed-discriminant}, we consider the optimization problem \eqref{eq:trace-opt} when $\mathcal M = \mathrm{pGr}(k,n)$. We count the critical points for any complex symmetric matrix $Q$ (Theorem~\ref{thm:defective}) and
describe the ED discriminant of the projection Grassmannian (Theorem ~\ref{thm:disc}). 
In Section \ref{sec:gd-degree}, we 
define the Grassmann distance (GD) degree of a subvariety of $\mathrm{pGr}(k,n)$, namely, the number of complex critical points of \eqref{eq:trace-opt} when the data matrix is a generic point within the Grassmannian. It is still unknown whether the GD degree is always a finite number. We prove that it is indeed finite for every model in $\mathrm{pGr}(2,n)$ (Theorem \ref{thm:gd-finite}).  
Finally, we show that the ED and GD degrees of $\mathcal M$ are expected to agree if and only if $\mathcal M$ has low dimension (Theorem~\ref{thm:EDGD}).

Section~\ref{sec:products} is devoted to computing the ED and GD degrees of products of smaller Grassmannians inside $\mathrm{pGr}(k,n)$, leading to Theorem \ref{thm:prod_Grassmannian}. As we will see, certain matroid and Schubert varieties are products of Grassmannians. 
In Section~\ref{sec:positroids}, we consider matroid varieties for matroids of rank $2$. 
We describe their degrees (Proposition~\ref{prop:matroid-degree}) and their singular loci (Proposition~\ref{prop:singular-rank-2}). 
We prove that their ED degrees satisfy a monotonicity property (Theorem \ref{thm:monotonicity}), and that their GD degrees are always powers of two (Theorem~\ref{thm:ed/gd-rank 2-positroids}). 
We conclude by listing the rank 3 matroid varieties on 6 elements up to matroid isomorphism, and giving their ED and GD degrees (Table~\ref{tab:36-matroids}). 
In Section~\ref{sec:schubert}, we show that all Schubert varieties in $\mathrm{pGr}(2,n)$ have GD degree $1$ or $2$ (Theorem~\ref{thm:k2-schub}), thereby resolving \cite[Conjecture 6.7]{FRS}. Finally, we prove Theorem~\ref{thm:gd-schubert} showing that all Schubert varieties defined by cyclically adjacent pivot columns have GD degree one. 
The code accompanying this paper is available~at: 
\begin{center}
\url{https://zenodo.org/records/22917921}
\end{center}

\section{Euclidean Distance Optimization}\label{sec:EDopt}
In this section, we treat the general Euclidean distance minimization problem \eqref{eq:trace-opt}
for subvarieties of the Grassmannian. In particular, we will extend a result about a form of genericity of $\mathcal{M}$ (ED-generality \cite{FRS}) from the setting of $\mathrm{Gr}(2,n)$ to $\mathrm{Gr}(k,n)$. We first briefly go over how one can compute the critical points of this optimization problem. 

For some computations, it is beneficial to work with Pl\"ucker coordinates or with a parametrization of $\mathcal M$ in the form of a $k \times n$ matrix, rather than the projection matrices \eqref{eq:pGr}.
We write $\mathrm{Gr}(k,n) \subset \mathbb P^{\binom n k - 1}$ for the Grassmannian in its Pl\"ucker embedding. 
Each point in a Zariski dense subset of $\mathrm{Gr}(k,n)$ can be identified with a point of $\mathrm{pGr}(k,n)$ as follows: let $X$ be the \emph{cocircuit matrix} (see, e.g., \cite{DFRS}) of a point in $\mathrm{Gr}(k,n)$.  
This is an $n \times \binom{n}{k-1}$ matrix whose $(i,J)$-th entry is the signed Pl\"ucker coordinate $x_{iJ}$
where $i \in [n]$ and $J \in \binom{[n]}{k-1}$.  By \cite[Corollary 2.5]{DFRS}, the corresponding $n \times n$ projection matrix $P$ is computed as 
\begin{equation*} 
    P = \frac{k}{{\rm trace}(XX^\top)}\, XX^\top.
\end{equation*}
This relation allows us to formulate the optimization problem \eqref{eq:trace-opt} using Pl\"ucker coordinates:
\begin{equation}\label{eq:trace-opt-over-Gr(k,n)}
    \textrm{maximize} \quad \phi_Q(X) \, := \, \,k \,\ \frac{{\rm trace}(XX^\top Q)}{{\rm trace}(XX^\top)} \quad \textrm{subject to} \quad X \in \mathcal{M}(\RR) \subset {\rm Gr}(k,n).
\end{equation}
Here, we view $\mathcal{M}$ as a subvariety of $\mathrm{Gr}(k,n) \subseteq \mathbb P^{\binom n k -1}$ via the entries of the cocircuit matrix.

Let $I_{\mathcal{M}}$ be the defining ideal of $\mathcal{M} \subset {\rm Gr}(k,n)$ minimally generated by $m$ polynomials. We assume that the codimension of $\mathcal{M}$ is $c$. Let also $\mathcal J_{\mathcal{M}}$ denote the $m \times \binom{n}{k}$ Jacobian matrix of $\mathcal{M}$. The augmented Jacobian matrix $\mathcal J^Q_{\mathcal{M}}$ is the $(m+2)\times \binom{n}{k}$ matrix 
\begin{equation}\label{eq:augJac}
    \mathcal J^Q_{\mathcal{M}}(X) = \left [ \begin{array}{c}
         \frac{\partial {\rm trace}(XX^\top Q)}{\partial x_I} \\ 
         x_I \\
         \mathcal J_{\mathcal{M}}
    \end{array}  
    \right ].
\end{equation}
Up to scaling, the first row in \eqref{eq:augJac} is the gradient of the numerator of $\phi_Q(X)$, and the second row is the gradient of the denominator. Then the critical ideal of \eqref{eq:trace-opt-over-Gr(k,n)} is 
\begin{equation}\label{eq:critideal}
    I^Q_{\mathcal{M}} \,\, = \,\,\biggl(I_{\mathcal{M}} \, \, + \,\,
\bigl\langle \,
(c\!+\!2) \times (c\!+\!2)\,\, \hbox{minors of}\,\,\, \mathcal{J}^Q_{\mathcal{M}}(X)\, \bigr\rangle \biggr) : \biggl(I_{\mathrm{sing}}(\mathcal{M}) \bigl \langle {\rm trace}(XX^\top)\bigr\rangle\biggr)^{\infty} 
\end{equation}
where $I_{\mathrm{sing}}(\mathcal{M})$ is the ideal of the singular locus of $\mathcal{M}$.
It is well established that, for generic $Q$, the critical ideal \eqref{eq:critideal} is zero-dimensional \cite{DHOST,FRS} with a constant number of complex zeros. 
The degree of $I^Q_{\mathcal M}$ is the ED degree of $\mathcal{M}$.  

Typically, the subvariety $\mathcal{M}$ will be parametrized by polynomials in $\RR[t_1,\ldots, t_s]$. In that case, the critical ideal is defined up to saturation by the $2\times 2$ minors of the matrix
\begin{align} \label{eq:crit-matrix}
\begin{bmatrix}
  \mathrm{trace}(X(t)X(t)^\top Q) & \frac{\partial \mathrm{trace}(X(t)X(t)^\top Q)}{\partial t_1} &  \frac{\partial \mathrm{trace}(X(t)X(t)^\top Q)}{\partial t_2} & 
  \cdots & \frac{\partial \mathrm{trace}(X(t)X(t)^\top Q)}{\partial t_s} \\
  \mathrm{trace}(X(t)X(t)^\top) & \frac{\partial \mathrm{trace}(X(t)X(t)^\top)}{\partial t_1} &  \frac{\partial \mathrm{trace}(X(t)X(t)^\top)}{\partial t_2} & 
  \cdots & \frac{\partial \mathrm{trace}(X(t)X(t)^\top)}{\partial t_s}
\end{bmatrix};
\end{align}
see  \cite[Section 3]{FRS}.
Here $X(t)$ denotes the cocircuit matrix written as a function of $t$.
\begin{example}
As a warm-up, we consider rational normal scrolls. For integers $2 \leq n_1 \leq n_2 \leq \cdots \leq n_d$, the rational normal scroll $S(n_1,\ldots, n_d)$ is a toric variety in $\PP^{n_1 + \cdots + n_d+d-1}$  that
 is ruled by the $(d-1)$-dimensional planes spanned by the rows of the matrix
$$ A(t) = \left[\begin{array}{ccccccccccccc} 1 & t & \cdots & t^{n_1} & 0 & 0 & \cdots & 0 & \cdots & 0 & 0 & \cdots & 0 \\
                          0 & 0 & \cdots & 0 & 1 & t & \cdots & t^{n_2} & \cdots & 0 & 0 & \cdots & 0 \\
\vdots & \vdots & \vdots & \vdots & \vdots & \vdots & \vdots & \vdots & \vdots  & \vdots  & \vdots & \vdots & \vdots\\
                          0 & 0 & \cdots & 0 & 0 & 0 & \cdots & 0 & \cdots & 1 & t & \cdots & t^{n_d}
\end{array} \right].$$
We consider the curve $\mathcal S$ in ${\rm Gr}(d,N+d)$ parametrized by $A(t)$, that is, the collection of the ruling subspaces. Here $N=n_1+\dots+n_d$. Let $\varphi$ denote Euler's totient function.
\end{example}
\begin{theorem}\label{thm:rat-norm}
    The curve $\mathcal S$ of $(d-1)$-planes ruling the rational normal scroll $S(n_1,\dots,n_d)$ has  ED degree $2\delta-2$ where 
    \begin{align*}
        \delta = \sum_{\substack{e >2 \\ \exists i: \, e |2(n_i+1) }} \varphi(e).
    \end{align*}
     In particular, if $\mathrm{gcd}(n_i+1,n_j+1) = 1$ for all $1 \leq i < j \leq d$, then this ED degree is $4N - 2$. 
\end{theorem}
\begin{proof}
    Let $v_i(t) = (1,t,\dots,t^{n_i})$ and $s_i(t) = 1+t^2+\dots+t^{2n_i}$. 
    By scaling the $i$-th row of $A(t)$ by $1/\sqrt{s_i(t)}$ we obtain $\bar{A}(t)$. The rows of $\bar{A}(t)$ form an orthonormal basis of the subspace they generate. 
    Therefore, the projection matrix onto this subspace is $P(t) = \bar{A}(t)^\top \bar{A}(t)$. This projection matrix is block diagonal, with $i$-th block of size $n_i+1$ equal to $\frac{v_i(t)^\top v_i(t)}{s_i(t)}$. Let $Q_1, \ldots, Q_d$ be the corresponding diagonal blocks of $Q$. Then the objective function in \eqref{eq:trace-opt} is
    \begin{align*}
        {\rm trace}(P(t)Q) = \sum_{i=1}^d \frac{f_i(t)}{s_i(t)}
    \end{align*}
     where $f_i(t):=v_i(t)Q_iv_i(t)^\top$. By genericity of $Q$, each $f_i$ is a generic polynomial of degree $2n_i$, and there are no relations between any $f_i$ and $f_j$ for $i\neq j$. Define $L(t):={\rm lcm}(s_1(t),\dots,s_d(t))$ and $\delta := {\rm deg}(L)$. Taking a common denominator we have
    \begin{align*}
        {\rm trace}(P(t)Q) = \frac{G(t)}{L(t)}, \quad G(t):=\sum_{i=1}^d f_i(t)\frac{L(t)}{s_i(t)}\,.
    \end{align*}
    Since ${\rm deg}(f_i)={\rm deg}(s_i)$ 
 we have ${\rm deg}(G)={\rm deg}(L)=\delta$. Now, for any root $\alpha$ of $L$ let $I_{\alpha}:=\{i \, : \, s_i(\alpha)=0\}$. Then $G(\alpha) = \sum_{i\in I_{\alpha}} f_i(\alpha)(\frac{L}{s_i})(\alpha)\neq 0$ generically, since $(\frac{L}{s_i})(\alpha)\neq 0$ for $i \in I_{\alpha}$. Hence $G$ and $L$ have no common factors. 

The numerator of the derivative of ${\rm trace}(P(t)Q)$ is $H(t) = G'(t)L(t) - G(t)L'(t)$, which has degree $2\delta-2$ since the leading term vanishes. If $L(\alpha)=0$, then $H(\alpha) = -G(\alpha)L'(\alpha) \neq 0$ since ${\rm gcd}(G,L)=1$ and $L$ is square-free.  It follows that $H$ has no zeroes at the poles of ${\rm trace}(P(t)Q)$. Let $u=\frac{1}{t}$. It can be shown that the $u$-derivative at $u=0$ of the objective function does not vanish generically, hence $H(t)$ has no zero at infinity. The critical points of the objective function thus correspond to the finite zeroes of $H$ and 
$\mathcal S$ has ED degree $2\delta - 2$.

    It remains to evaluate $\delta$. The polynomial $s_i$ has the cyclotomic factorization
    \begin{align*}
        s_i(t) = \frac{t^{2(n_i+1)}-1}{t^2-1} = \prod_{\substack{e>2 \\ e|2(n_i+1)}} \Phi_e(t)\,,
    \end{align*}
    where $\Phi_e$ is the $e$-th cyclotomic polynomial. Since all the factors have multiplicity $1$ and ${\rm deg}(\Phi_e)=\varphi(e)$, the statement follows. In particular, if ${\rm gcd}(n_i+1,n_j+1)=1$ for all $1\leq i<j\leq d$, and if some $e>2$ divides $2(n_i+1)$, then it cannot divide any other $2(n_j+1)$. The result follows from $\sum_{e | M}\varphi(e) = M$ for every $M \in \mathbb Z_{>0}$. 
\end{proof}

\subsection{ED-generality}
The critical ideal $I_\mathcal{M}^Q$ can be computed by saturating only with respect to $I_{\mathrm{sing}}(\mathcal{M})$ (and hence with no saturation if $\mathcal{M}$ is smooth) if the hypersurface $\{\mathrm{trace}(XX^\top)=0 \}$ intersects $\mathcal{M}$ transversally. In this case, the subvariety $\mathcal{M}$ is called ED-general 
\cite[Section 3]{FRS}. Theorem 3.3 in \cite{FRS} gives a criterion on ED-generality of $\mathcal{M} \subset \mathrm{Gr}(2,n)$ based on the dimension of this subvariety. Now we generalize this result to subvarieties of $\mathrm{Gr}(k,n)$. 

\begin{theorem} \label{thm:codim-sing-locus} The codimension of the singular locus of 
${\rm Gr}(k,n) \cap \{\mathrm{trace}(XX^\top)=0\}$ is at most $\binom{k+1}{2}$. Therefore, no subvariety $\mathcal{M}$ of ${\rm Gr}(k,n)$ with $\dim(\mathcal{M}) \geq \binom{k+1}{2}$ is ED-general.     
\end{theorem}
We will prove this theorem using Proposition \ref{prop:singular-locus-transversality} below which is of independent interest. For this purpose, we study the singular locus of ${\rm Gr}(k,n) \cap \{\mathrm{trace}(XX^\top)=0\}$ more closely.

Let $A = [\mathrm{Id}_k \,\,\, B]$ be the $k \times n$ matrix parametrizing a chart of ${\rm Gr}(k,n)$ where $B = (b_{ij})$ with $b_{ij}$ the entry in row $i=1,\ldots, k$ and column $j=k+1, \ldots, n$ of $A$. For every $I \subset [k]$ and $J \subset [n] \setminus [k] = \{k+1, \ldots, n\}$ with $|I| = |J|$, let $B_{I,J}$ be the submatrix of $A$ using rows and columns indexed by $I$ and $J$. The morphism $\CC^{k(n-k)} \to {\rm Gr}(k,n) \cap \{x_{[k]} =1\}$
given by 
\[A \mapsto \bigl(\det(B_{I,J}) \, :\, I \subset [k], \, J \subset [n] \setminus [k], \, |I|=|J|\bigr)\]
where we set $\det(B_{\emptyset, \emptyset}) =1$ is an isomorphism. This is true for any other affine chart of ${\rm Gr}(k,n)$ given by $x_K =1$ where now the submatrix of $A$ consisting of the columns indexed by $K$ is the identity matrix $\mathrm{Id}_k$. 
Under this isomorphism, the image of the vanishing locus~of
 $$  F_{k,n} := \sum_{\substack{I \subset [k], J \subset [n] \setminus [k] \\ |I| = |J|}} \det(B_{I,J})^2 = \sum_{I \subset [k]} \sum_{\substack{J \subset [n] \setminus [k] \\ |J| = k - |I|}} \det(B_{[k] \setminus I, J})^2$$
 is the codimension one subvariety ${\rm Gr}(k,n) \cap \{\mathrm{trace}(XX^\top)=0\}$
in the chart given by $x_{[k]}=1$.
Thus, the singular locus of ${\rm Gr}(k,n) \cap \{\mathrm{trace}(XX^\top) = 0\}$ in the affine chart $x_{[k]}=1$ is the isomorphic image of the singular locus of $V(F_{k,n})$, 
which is defined by the ideal 
$$I_{\mathrm{sing}}(k,n) = \langle F_{k,n}\rangle + \biggl \langle  \frac{\partial F_{k,n}}{\partial b_{ij}} \,: \: i\in [k], j\in [n] \setminus [k] \biggr \rangle.$$  
Let $X$ be the cocircuit matrix of $A$ and let $H(k,n)$ be the ideal generated by the $\binom{k+1}{2}$ unique entries of
the leading principal $k\times k$ (symmetric) 
submatrix of $XX^\top$. For $1 \leq i \leq j \leq k$, the $(i,j)$ entry of this submatrix is
$$H_{ij} \, := \, \sum_{S \subset [k] \setminus \{i,j\}}  (-1)^{|S_{ij}|} \left( \sum_{\substack{J \subset [n] \setminus [k] \\ |J|= k-|S|-1}} \det(B_{[k]\setminus \{S \cup i\}, J}) \det(B_{[k]\setminus \{S \cup j\},J})\right)$$
where $S_{ij} = \{s \in S \, :\, i < s <j\}$.
With this, we define $H(k,n) = \langle H_{ij} \, : \, 1\leq i \leq j \leq k\rangle$.
\begin{proposition} \label{prop:singular-locus-transversality}
The singular locus of $V(F_{k,n})$ is generated by $H_{ij}$ for $1 \leq i \leq j \leq k$: 
\[I_{\mathrm{sing}}(k,n) = H(k,n).\]
\end{proposition}
\begin{proof}
One can show that 
$$ H_{ii} = F_{k,n} - \frac{1}{2} \sum_{t=k+1}^n b_{it}\frac{\partial F_{k,n}}{\partial b_{it}} \quad \mbox{  and  } \quad H_{ij} = -\frac{1}{2} \sum_{t=k+1}^n b_{jt}\frac{\partial F_{k,n}}{\partial b_{it}}.$$
For instance, observe that $\frac{\partial \det(B_{[k]\setminus I, J})^2}{\partial b_{it}} = 0$ for all $t \in [n]\setminus [k]$ if $i \in I$. Moreover, if $ i \not \in I$, $\sum_{t=k+1}^n b_{it} \frac{\partial \det(B_{[k]\setminus I,J})^2}{\partial b_{it}} = 2 \det(B_{[k]\setminus I,J})^2$.
Therefore 
\[F_{k,n} - \frac{1}{2} \sum_{t=k+1}^n b_{it}\frac{\partial F_{k,n}}{\partial b_{it}} = \sum_{\substack{I \subset [k] \\ i \in I}} \sum_{\substack{J \subset [n] \setminus [k] \\ |J| = k - |I|}} \det(B_{[k] \setminus I, J})^2 = \sum_{S \subset [k] \setminus \{i\}}   \sum_{\substack{J \subset [n] \setminus [k] \\ |J|= k-|S|-1}}\!\!\!\! \det(B_{[k]\setminus \{S \cup i\}, J})^2 = H_{ii}.\]
This shows $H(k,n) \subset I_{\mathrm{sing}}(k,n)$.

For the converse, we first show that $F_{k,n} \in H(k,n)$. This follows from the fact that the cocircuit matrix $X$ has rank $k$, and in the chart $x_{[k]}=1$ the first $k$
rows are linearly independent. We can write the rest of its rows explicitly as a linear combination of the first $k$ rows. In fact, if we let $X_i$ to be the $i$-th row of $X$, for $t\in[n] \setminus [k]$, we have $X_t = \sum_{j=1}^k b_{jt} X_j$. Now since $k F_{k,n} = X_1X_1^\top + \cdots + X_nX_n^\top$  and since $H_{ij} = X_iX_j^\top$, we conclude that $F_{k,n} \in H(k,n)$.
Next we need to argue that all partial derivatives of $F_{k,n}$ are in $H(k,n)$.  We will show that 
$$\frac{\partial F_{k,n}}{\partial b_{it}} = 2\sum_{j=1}^k b_{jt}H_{ij} = 2 \sum_{j=1}^k b_{jt} X_iX_j^\top = 2 X_i \left(\sum_{j=1}^k b_{jt} X_j\right)^\top. $$
First, we compute that
\begin{equation} \label{eq:partial_F}
\frac{\partial F_{k,n}}{\partial b_{it}} = 2 \sum_{\substack{I \subset [k] \\ i \not \in I}} \sum_{\substack{J \subset [n] \setminus [k] \\ |J| = k- |I| \\ t \in J}} \mathrm{sign}(i, [k] \setminus I) \mathrm{sign}(t,J) \det(B_{[k] \setminus I, J}) \det(B_{[k] \setminus \{I \cup i\}, J \setminus t})
\end{equation}
where $\mathrm{sign}(\ell, S)$ is $(-1)^r$ if $\ell \in S$ is in position $r$ of the ordered set $S$. Next we compute the entry of $\sum_{j=1}^k b_{jt} X_j$ indexed by the subset $I \cup J \subset [n]$ of size $k-1$ where $I \subset [k]$ and $J \subset [n] \setminus [k]$. This entry is equal to 
$$ \sum_{j=1}^k b_{jt} \det(A_{[k], j \cup I \cup J}) = \mathrm{sign}(t, J\cup t) \det(B_{[k] \setminus I, J \cup t})$$
whenever $t \not \in J$ and $0$ otherwise. 
The entry in $X_i$ indexed by the same $I \cup J$ is zero if $i \in I$; otherwise it is equal to $ \mathrm{sign}(i, [k] \setminus I)\det(B_{[k] \setminus \{I \cup i\}, J})$. Putting this together we conclude that 
$$ 2 X_i \left(\sum_{j=1}^k b_{jt} X_j\right)^\top = 2 \sum_{\substack{I \subset [k] \\ i \not \in I}} \sum_{\substack{J \subset [n] \setminus [k] \\ |J| = k- |I|-1 \\ t \not \in J}} \mathrm{sign}(i, [k] \setminus I) \mathrm{sign}(t,J \cup t) \det(B_{[k] \setminus I, J \cup t}) \det(B_{[k] \setminus \{I \cup i\}, J}).$$
This expression is equal to \eqref{eq:partial_F}.
Hence $I_{\mathrm{sing}}(k,n) \subset H(k,n)$ and the result is proved. 
\end{proof}

\begin{proof}[Proof of Theorem \ref{thm:codim-sing-locus}] Proposition~\ref{prop:singular-locus-transversality} holds for any affine chart of ${\rm Gr}(k,n)$ given by setting a Pl\"ucker coordinate equal to one. Therefore, an affine open cover of the singular locus of ${\rm Gr}(k,n) \cap \{\mathrm{trace}(XX^\top)=0\}$ is given by the isomorphic image of each $V(H(k,n))$ under the Pl\"ucker embedding. By Krull's Principal Ideal Theorem, these affine varieties have codimension at most $\binom{k+1}{2}$, implying that the singular locus
of ${\rm Gr}(k,n) \cap \{\mathrm{trace}(XX^\top)=0\}$ also has at most this codimension. 
If $\dim(\mathcal M) \geq \binom{k+1}{2}$, then $\mathcal M$ intersects this singular locus and at these intersection points the augmented Jacobian matrix \eqref{eq:augJac} drops rank, regardless of what the model or data are.
Since these points are independent of the data, they form an irreducible component of the degeneracy locus of \eqref{eq:augJac}. Thus the model is not ED-general.
\end{proof}

\begin{corollary}[{\cite[Theorem 3.3]{FRS}}] 
No subvariety $\mathcal{M}$ of ${\rm Gr}(2,n)$  with 
$\dim(\mathcal{M}) \geq 3$ is ED-general. If $\dim(\mathcal{M}) < 3$ and $\mathcal{M}$ is generic, then $\mathcal{M}$ is ED-general.
\end{corollary}
\begin{proof}
The ideal $H(2,n)$ is generated by the three polynomials $H_{11}, H_{22}, H_{12}$. These form a regular sequence, and thus $H(2,n)$ has codimension $3$.
\end{proof}

\section{ED Discriminant of the Grassmannian} \label{sec:ed-discriminant}
In the previous section, we reviewed the Euclidean distance optimization problem for subvarieties in the Grassmannian. 
We will now focus on the 
Euclidean distance optimization problem \eqref{eq:trace-opt} on $\mathcal{M} = \mathrm{pGr}(k,n)$: 
\begin{align}\label{eq:proj-optprob}
 \textrm{maximize} \quad \mathrm{trace}(PQ) \quad \textrm{subject to} \quad P \in \mathrm{pGr}(k,n).   
\end{align}
 
 \begin{remark} \label{rem:O(n)-equivariant} 
 We note that \eqref{eq:proj-optprob} is $\mathrm{O}(n)$-equivariant
 where the complex orthogonal group $\mathrm{O}(n)$ acts on $\mathrm{pGr}(k,n)$ by conjugation. This action is transitive by \cite[Theorem XI.4]{gantmacher}.  
 \end{remark}
 We will study the set of  data points $Q$ which have a critical point set of cardinality different than the ED degree.
 These special data points  comprise the \emph{ED discriminant} \cite[Section 7]{DHOST} of $\mathrm{pGr}(k,n)$. In this section, we describe this ED discriminant  and give the number of critical points depending on the number of Jordan blocks of $Q$.

We first describe the critical points of the problem \eqref{eq:proj-optprob} for a general data matrix $Q$. The result below and its proof are straightforward generalizations of those in \cite[Proposition 4.1]{FRS}. 

\begin{proposition}\label{prop:cps}
    The critical point set of \eqref{eq:proj-optprob} is 
    \begin{align*}
    \{P \in \mathrm{pGr}(k,n): PQ = QP \} = 
        \{P \in \mathrm{pGr}(k,n): 
            P \textrm{ projects onto a $Q$-invariant subspace}
        \}.
    \end{align*}
\end{proposition}
\begin{proof}
We first argue that the two sets are equal.
Suppose $P$ projects onto a $Q$-invariant subspace. 
Then for any vector $v \in \mathbb C^n$, the identity $PQPv = QPv$ holds, and thus $PQP = QP$ holds as a matrix identity. 
Taking the transpose of both sides gives the result. 
Conversely, if $PQ = QP$, then $v \in \mathrm{im}(P)$ implies $Qv = QPv = P(Qv) \in \mathrm{im}(P)$, so $\mathrm{im}(P)$ is $Q$-invariant.

We now show that this set is the critical point set of \eqref{eq:proj-optprob}.
Since the optimization problem \eqref{eq:proj-optprob} is $\mathrm{O}(n)$-equivariant, 
 without loss of generality, we may assume that
 $P = \mathrm{Id}_k \oplus 0_{{n-k}}$.
 By \cite[Section 2.3]{BZA}, the tangent space at $P$ is $T_P\mathrm{pGr}(k,n) = \big\{\begin{bmatrix}
        0_{k} & A \\
        A^\top & 0_{n-k}
    \end{bmatrix} : A\in \mathbb C^{k \times (n - k)} \big\}$. 
    The point $P$ is critical for $Q$ if and only if $\mathrm{trace}(\dot PQ) = 0$ for all $\dot P \in T_P\mathrm{pGr}(k,n)$.
    Writing $Q = 
    \begin{bmatrix}
        Q_{11} & Q_{12}\\ Q_{12}^\top & Q_{22}
    \end{bmatrix}$, 
    we have that $P$ is critical for $Q$ if and only if $Q_{12} = 0$. 
    The matrices $Q$ of the form  $Q = Q_{11}\oplus Q_{22}$ are precisely the ones commuting with $P$.
\end{proof}

For a generic symmetric matrix $Q$, the number of $k$-dimensional $Q$-invariant subspaces is $\binom{n}{k}$, each corresponding to a choice of $k$ eigenvectors; see \cite{FH, LLY}. Therefore, by Proposition~\ref{prop:cps}, the Euclidean distance optimization problem \eqref{eq:proj-optprob} has $\binom{n}{k}$ critical points on ${\rm pGr}(k,n)$.
However, there exist special data matrices that do not have $\binom{n}{k}$ critical points. The Zariski closure of these special data is known as the ED discriminant \cite{DHOST} of $\mathrm{pGr}(k,n)$. 
In symbols,
\begin{align*}
    \Sigma = 
    \overline{
    \{Q \in \mathrm{Sym}^2\CC^n: 
    \textrm{the critical point set of \eqref{eq:proj-optprob} does not consist of  $\binom{n}{k}$ reduced points}\}}.
\end{align*}
Suppose first that $Q$ is diagonalizable. 
If the eigenvalues of $Q$ have multiplicity $1$, then \eqref{eq:proj-optprob}  has $\binom{n}{k}$ critical points and $Q \notin \Sigma$.
On the other hand, if $Q$ has an eigenvalue with higher multiplicity, then $Q$ has infinitely many $k$-dimensional invariant subspaces for $1 \leq k \leq n-1$, and thus $Q \in \Sigma$. 
This suggests our main result in this section which is equivalent to Theorem~\ref{thm:disc1}. 
\begin{theorem} \label{thm:disc}
The ED discriminant $\Sigma$ of $\mathrm{pGr}(k,n)$ is equal to the discriminant
\begin{align*}
    \Delta = 
    \{Q \in \mathrm{Sym}^2\mathbb C^n: 
    \det(Q - z\cdot \mathrm{Id}_n) \textrm{ has a multiple root}\}.
\end{align*}
\end{theorem}
Note that the discriminant of the characteristic polynomial of a complex symmetric matrix has an explicit expression
\cite{Ily92}, and this expression is a sum of squares \cite[Theorem~7.11]{Stu02}. 

The argument above proves the result for diagonalizable matrices.
Since diagonalizable complex symmetric matrices are diagonalizable by complex orthogonal matrices \cite[Corollary XI.3.2]{gantmacher}, a dimension count shows that they are not Zariski dense in $\Delta$,
so our argument above is not sufficient to prove Theorem \ref{thm:disc}.
We must examine the more subtle case of defective matrices.
The first observation we make in the defective case is that not every invariant subspace of $Q$ has a symmetric projection. 
If the rows of a $k \times n$ matrix $A$ span a $Q$-invariant subspace, then the symmetric projection onto the space is $P = A^\top(AA^\top)^{-1}A$. 
The matrix $P$ exists if and only if $AA^\top$ has full rank if and only if the row span of $A$ has a basis with no isotropic vectors.  
The next result shows that the invariant subspaces that have a symmetric projection matrix are those that do not break up Jordan chains.  
 
\begin{proposition} \label{prop:gen-eigen}
    Let $Q,P \in \mathrm{Sym}^2\mathbb C^n$ be such that $P^2 = P$ and $QP = PQ$. 
    Then $Q$ gives a $\CC^n$-basis of generalized eigenvectors such that every Jordan chain is in $\mathrm{im}(P)$ or $\mathrm{ker}(P)$.
\end{proposition}
\begin{proof}
    Suppose $\mathrm{rank}(P) = k$. 
    Since $\mathrm{O}(n)$ acts transitively on $\mathrm{pGr}(k,n)$ and $QP = PQ$ is invariant under this action, we assume that $P = \mathrm{Id}_k \oplus 0_{n-k}$.
    Then $Q$ is block diagonal with a $k \times k$ block and an $(n-k) \times (n-k)$ block. 
    The result follows from the Jordan canonical form for each individual block and taking the union of the bases of generalized eigenvectors. 
\end{proof}

Our next result will describe the critical points set of \eqref{eq:proj-optprob} for any symmetric matrix $Q$.
For the remainder of the section, we let $\lambda_1, \ldots, \lambda_\ell$ be the  distinct eigenvalues of $Q$. 
For each $i = 1, \ldots, \ell$, let $S_i$ denote the multiset of sizes of Jordan blocks with eigenvalue $\lambda_i$. 
The Jordan blocks are therefore indexed by the disjoint union $\mathcal P = \bigsqcup_{i=1}^\ell S_i$.

 By Propositions~\ref{prop:cps} and~\ref{prop:gen-eigen}, we associate to every critical point of \eqref{eq:proj-optprob} a subset $\mathcal S$ of $\mathcal P$ such that $\sum_{s \in \mathcal S} s = k$ where the elements of $\mathcal S$ are the lengths of the Jordan chains in the image of $P$ in Proposition~\ref{prop:gen-eigen}. 
    We will prove that for each possible choice of $\mathcal S$, the problem \eqref{eq:proj-optprob} has either $1$ or infinitely many critical points associated to $\mathcal S$.
    In the finite case, this reduces the problem of counting the critical points to counting the possible sets $\mathcal S$.
    
\begin{theorem}\label{thm:defective}
    Let $Q \in \mathrm{Sym}^2 \CC^n$ have Jordan blocks with sizes in $\mathcal P$.
    The number of critical points of \eqref{eq:proj-optprob} is finite if and only if every set $\mathcal S \subseteq \mathcal P$ with $\sum_{s \in \mathcal S} s = k$ has the form $\mathcal S = \bigsqcup_{i \in \mathcal T} S_i$ for some $\mathcal T \subseteq [\ell]$.
    In this case, the number of critical points is equal to the coefficient of $x^k$ in the polynomial $\prod_{s \in \mathcal P} (1 + x^s).$
\end{theorem}
\begin{proof}
    Suppose $\mathcal S \subseteq \mathcal P$ satisfies $\sum_{s \in \mathcal S} s = k$.
    We first argue that there is at least one critical point whose image contains Jordan chains of $Q$ with lengths in $\mathcal S$.
    Indeed, by \cite[Corollary XI.3.2]{gantmacher}, there exists a complex orthogonal matrix $R$ and a block diagonal matrix $B$ whose block sizes are given by $\mathcal P$ such that $Q = RBR^\top$. 
    Choose blocks with sizes in $\mathcal S$; the corresponding rows of $R$ span a $Q$-invariant subspace and since $R$ is orthogonal, there exists a symmetric projection matrix $P$ of rank $k$ onto this vector space. 
    
    If $\mathcal S = \bigsqcup_{i \in \mathcal T}S_i$ for some $\mathcal T \subseteq [\ell]$, then 
    by Proposition \ref{prop:gen-eigen}, any critical point of \eqref{eq:proj-optprob} projects onto the sum of the generalized eigenspaces for eigenvalues $\lambda_i$ with $i \in \mathcal T$.
    This point is uniquely determined by the set $\mathcal S$. 
    If every set $\mathcal S$ with $\sum_{s \in \mathcal S}s = k$ has the form $\mathcal S = \bigsqcup_{i \in \mathcal T}S_i$, then there is one critical point for each set. 
    The number of such sets is counted by the coefficient in the statement.

    Suppose there exists an $i$ such that $ S_i\not \subseteq \mathcal S, \mathcal P\backslash \mathcal S$.
    By Proposition~\ref{prop:gen-eigen}, there is a basis $\mathcal B$ for the image of $P$ consisting of Jordan chains, so there exists a $Q$-Jordan chain $v_1, \ldots, v_{r} \in \mathcal B$ and a $Q$-eigenvector $u \in \mathrm{ker}(P)$ both with eigenvalue $\lambda_i$.
    Then $v_1, \ldots, v_{r-1} , v_{r} + \alpha \cdot u$ is a Jordan chain for any $\alpha \in \mathbb C$. 
    One verifies that the Gram matrix $G'$ of the new Jordan chain has the same rank as the Gram matrix $G$ of $v_1, \ldots, v_r$ for almost all $\alpha$. 
    Since $G$ has full rank, $G'$ does as well, and thus there exists a symmetric projection matrix onto the span of $(\mathcal B\backslash v_r) \cup (v_r + \alpha u)$.
    We conclude that there are infinitely many critical points of \eqref{eq:proj-optprob}.
\end{proof}

\begin{example}
    The table shows critical point counts of \eqref{eq:proj-optprob} for $n = 6$ when all eigenvalues of $Q$ have geometric multiplicity $1$. 
    The columns are indexed by Jordan block sizes. 
       \[
    \begin{array}{|c|ccccccccccc|}
    \toprule    
        k & 6^1 & 5^11^1 & 4^12^1 & 4^11^2 & 3^2 & 3^12^11^1 & 3^11^3 & 2^3 & 2^21^2 & 2^11^4 & 1^6 \\
        \midrule
        1 & 0 & 1 & 0 & 2 & 0 & 1 & 3 & 0 & 2 & 4 & 6\\ 
        2 & 0 & 0 & 1 & 1 & 0 & 1 & 3 & 3 & 3 & 7 & 15\\
        3 & 0 & 0 & 0 & 0 & 2 & 2 & 2 & 0 & 4 & 8 & 20\\
        \bottomrule
    \end{array}
    \]
\end{example}

\begin{proof}[Proof of Theorem \ref{thm:disc}:]
    If $Q \notin \Delta$, then $Q$ has $n$ distinct eigenvalues and so by  Theorem~\ref{thm:defective}, 
    the optimization problem \eqref{eq:proj-optprob} has $\binom n k$ isolated, reduced critical points. 
    Thus  $\Sigma \subseteq \Delta$.
    Conversely, suppose $Q \in \Delta$.
    If there is an element of $\mathcal P$ which is greater than $1$, then the coefficient of $x^k$ in Theorem~\ref{thm:defective} is less than $\binom{n}{k}$. 
    If all elements of $\mathcal P$ equal to $1$, then $P$ is diagonalizable and
    \eqref{eq:proj-optprob} has infinitely many critical points by Theorem~\ref{thm:defective}.
    Thus $\Delta \subseteq \Sigma$.
\end{proof}
We remark that the Grassmannian $\mathrm{pGr}(k,n)$ is entirely contained within the ED discriminant, since projection matrices have only two distinct eigenvalues. 
Since all points in $\mathrm{pGr}(k,n)$ are diagonalizable, Theorem~\ref{thm:defective} implies that \eqref{eq:proj-optprob} has infinitely many critical points if $Q \in \mathrm{pGr}(k,n)$. 
We describe these in detail in the next section.

\section{Cut Locus and GD Degree}\label{sec:gd-degree}

The aim of this section is to describe the critical points of \eqref{eq:proj-optprob} when the data matrix $Q$ is generic in the Grassmannian $\mathrm{pGr}(k,n)$.
This problem will be important for our definition of the GD degree. 
Suppose that $X$ is a critical point in Pl\"ucker coordinates of \eqref{eq:proj-optprob}. 
This implies that the augmented Jacobian 
\begin{align*}
    \mathcal J_{\mathrm{Gr}(k,n)}^Q(X) = 
    \begin{bmatrix}
        \frac{\partial \mathrm{trace}(XX^\top Q)}{\partial x_I}\\
        x_I\\
        \mathcal J_{\mathrm{Gr}(k,n)}
    \end{bmatrix}
\end{align*}
drops rank. 
Let $\mathcal M \subseteq \mathrm{pGr}(k,n)$ be a model.
If $\mathcal J_{\mathrm{Gr}(k,n)}^Q(X)$ is rank deficient, the matrix $\mathcal J_{\mathcal M}^Q(X)$ is rank deficient, too, regardless of what $\mathcal M$ is. 
If, in addition, $X \in \mathcal M$, then $X$ will be in the vanishing locus of the critical ideal \eqref{eq:critideal} of $Q$.
Critical points that arise in this way are extraneous. 
To be precise, we define the \emph{extraneous critical points} of $Q$ on $\mathcal M$ as 
\begin{equation*}
    \mathcal{E}_\mathcal M^Q = 
    \overline{\{P \in \mathcal M_{\mathrm{reg}} : P \textrm{ is critical for } \eqref{eq:proj-optprob}\}}
\end{equation*}
where $\mathcal{M}_{\mathrm{reg}}$ is the set of nonsingular points of $\mathcal{M}$.
We remove these from our critical point set, because they are independent of the model. 
The next result shows that we can remove these extraneous critical points by removing the \emph{algebraic cut locus} of $Q$:
\begin{equation} \label{eq:cut}
    \mathrm{cut}(Q) := \{P \in \mathrm{pGr}(k,n): \mathrm{rank}(QP) < k\} = 
    \{P \in \mathrm{pGr}(k,n) : \mathrm{ker}(Q) \cap \mathrm{im}(P) \neq 0\}.
\end{equation}
We observe that $P \in \mathrm{cut}(Q)$ if and only if $Q \in \mathrm{cut}(P)$.
Note that if $X$ and $Y$ are the cocircuit matrices of $P$ and $Q$, then the algebraic cut locus may be identified with the vanishing set of ${\rm trace}(XY^\top)$; that is the definition given in \cite{FRS}.
\begin{lemma}
    For generic $Q   \in \mathrm{pGr}(k,n)$, the extraneous critical points $\mathcal{E}_\mathcal M^Q $ are contained in the algebraic cut locus of $Q$. 
\end{lemma}
\begin{proof}
    If $P \in \mathcal E_{\mathcal M}^Q$ is critical for \eqref{eq:proj-optprob}, then $P$ projects onto a $Q$-invariant subspace. Since $Q$ is generic, $Q$ is not contained in $\mathcal E_{\mathcal M}^Q \subseteq \mathcal M$, and hence $P \neq Q$.
    Because $Q$ has rank $k$, $P$ must have a kernel vector of $Q$ in its image. Thus $P \in \mathrm{cut}(Q)$.
\end{proof}
For sufficiently general models, $\mathcal E_{\mathcal M}^Q = \mathcal M \cap \mathrm{cut}(Q)$.  
In this case, we call the model $\mathcal{M}$ \emph{GD-general} \cite{FRS}.
We now define the Grassmann distance degree of a model $\mathcal M \subseteq \mathrm{pGr}(k,n)$. 
\begin{defprop}\label{def:GD}
    Let $\mathcal M$ be an irreducible subvariety of $\mathrm{pGr}(k,n)$ with Zariski dense real points. 
    For generic $Q \in \mathrm{pGr}(k,n)$, the critical set
    \begin{equation}\label{eq:crit}
        \overline{
        \{P \in \mathcal M_{\mathrm{reg}} \backslash \mathrm{cut}(Q) : 
        \mathrm{rank}\, \mathcal J_{\mathcal M}^Q \leq \mathrm{codim}(\mathcal M)+ 1 
        \}}
    \end{equation}
    is either infinite or has a fixed size which does not depend on $Q$.
    The \emph{Grassmann distance (GD) degree} of $\mathcal M$ is the size of the set \eqref{eq:crit}. 
\end{defprop}
To prove that the GD degree is well-defined, we define the \emph{GD correspondence} of model points $P$ that are critical for data points $Q$: 
\begin{equation*}
    \mathcal I_{\mathcal M} = 
     \overline{
        \{(P,Q) \in \mathcal M_{\mathrm{reg}} \backslash \mathrm{cut}(Q) \times \mathrm{pGr}(k,n): 
        \mathrm{rank}\, \mathcal J_{\mathcal M}^Q \leq \mathrm{codim}(\mathcal M)+ 1 
        \}} \subseteq \mathcal M \times \mathrm{pGr}(k,n).
\end{equation*}
Let $\pi_1: \mathcal I_\mathcal M \to \mathcal M$ and $\pi_2: \mathcal I_\mathcal M \to \mathrm{pGr}(k,n)$ denote the natural projections.
\begin{proof}
     The set \eqref{eq:crit} is $\pi_1(\pi_2^{-1}(Q))$.
    The morphism $\pi_2$ is dominant, so at least one irreducible component of $\mathcal I_\mathcal M$ maps dominantly to $\mathrm{pGr}(k,n)$. 
    For every such component, the fibers have expected dimension on a dense open set \cite[Theorem 12.4.1]{vakil}, and, if the map is generically finite, it has locally constant degree \cite[24.5.14.]{vakil}.
    The proof is completed by constructing an open set of $\mathrm{pGr}(k,n)$ on which the GD degree is constant. 
\end{proof}
The \emph{ED correspondence} \cite[Section 4]{DHOST} is defined similarly to $\mathcal I_\mathcal M$, but with $(P, Q) \in \mathcal M_{\mathrm{reg}} \times \mathrm{Sym}^2\mathbb C^n$.
The ED correspondence is a vector bundle over $\mathcal M_\mathrm{reg}$, and it is always irreducible of dimension $\binom{n+1}{2}$. 
Since the ED correspondence has the same dimension as the data space $\mathrm{Sym}^2\mathbb C^n$, the ED degree is finite. 
The GD correspondence, on the other hand, can be reducible, as we show in the next example.
However, we show that, at least for $k=2$, all components of $\mathcal I_{\mathcal M}$ that map dominantly onto $\mathrm{pGr}(k,n)$ have dimension $2(n-2) = \mathrm{dim}(\mathrm{pGr}(2,n))$. 
This implies that all models in $\mathrm{pGr}(2,n)$ have finite GD degree.

\begin{example}\label{ex:nonirred}
    Consider the model in $\mathrm{Gr}(2,4)$ parametrized by matrices
    \begin{align*}
        A = \begin{bmatrix}
            1 & 0 & t & 0\\
            0 &  1 & 0 & t 
        \end{bmatrix}.
    \end{align*}
    In Pl\"ucker coordinates, the ideal of this model is $I = \langle x_{13}, x_{24}, x_{14} + x_{23},  x_{12}x_{34} + x_{14}x_{23} \rangle$.
    We compute the incidence correspondence as detailed above using Pl\"ucker coordinates $x$ for the model matrix and $y$ for the data matrix. The ideal of the GD correspondence in the ring $\mathbb R[x,y]/(I + \langle y_{12}y_{34} - y_{14}y_{23}+ y_{14}y_{23} \rangle )$ has prime decomposition
    \begin{equation*}
        \langle 
         y_{12} + y_{34} \rangle 
        \cap 
        \big 
         \langle 
         \det\begin{bmatrix}
         x_{12} - x_{34} & x_{14} - x_{23}\\
         y_{12} - y_{34} & y_{14} - y_{23}\\
         \end{bmatrix}
         \big \rangle. 
    \end{equation*}
    Note that the component defined by $y_{12} + y_{34}$ does not project dominantly onto $\mathrm{Gr}(2,n)$. 
\end{example}
\noindent
We now show that every model in $\mathrm{pGr}(2,n)$ has finite GD degree.
\begin{proof}[Proof of Theorem~\ref{thm:gd-finite}]
    Assume for contradiction that the generic fiber of $\pi_2 :\mathcal{I}_{\mathcal{M}} \rightarrow {\rm pGr}(2,n)$ is positive dimensional. Then there exists an irreducible component $\mathcal C$ of $\mathcal{I}_{\mathcal{M}}$ such that   $\pi_2|_{\mathcal C}$  is dominant and ${\rm dim}(\mathcal C)>{\rm dim}({\rm pGr}(2,n))=2(n-2)=:d$. 
    Define $F: \mathcal M \times \mathrm{pGr}(k,n) \to \mathbb C$ by 
    $F(P,Q)={\rm trace}(PQ)$ and denote 
    \begin{equation*}
    \psi = (\pi_2 \times F)|_{\mathcal C}:\mathcal C \rightarrow {\rm pGr}(2,n)\times \mathbb C, \quad \quad  \psi(P,Q)=(Q,F(P,Q)).
    \end{equation*}
    
    We now show that $\psi$ has positive dimensional fibers.
    At a smooth point $(P,Q)\in\mathcal C$ we have ${\rm trace}(UQ)=0$ for every $U \in T_P\mathcal M$.
    By the product rule the differential of $F$ is \[
    d_{(P,Q)}F(U,V) = {\rm trace}(UQ) + {\rm trace}(PV) = {\rm trace}(PV) \,,
    \]
    and hence $d_{(P,Q)}\psi(U,V)=(V, {\rm trace}(PV))$ for every $(U,V)\in T_{(P,Q)}\mathcal C$. 
    For fixed $P$, the linear map $ d \pi_2(T_{(P,Q)}\mathcal C) \to T_Q\mathrm{pGr}(k,n) \times \mathbb C ,V \mapsto (V,{\rm trace}(PV))$ is injective and ${\rm rank}(d_{(P,Q)}\psi) = {\rm rank}(d_{(P,Q)}\pi_2|_{\mathcal C})=d$. 
    Let $\mathcal H:=\overline{\psi(\mathcal C)} \subseteq \mathrm{pGr}(k,n) \times \mathbb C$. By generic smoothness of algebraic morphisms \cite[Theorem 21.6.4]{vakil}, ${\rm dim}(\mathcal H)=d$ and, since $\dim(\mathcal C) > d$ by assumption, the generic fiber of $\psi$ is positive dimensional.

    Since $\mathrm{rank}(d_{(P,Q)}\psi) = d$ at generic points in $(P,Q) \in \mathcal C$, the differential 
    \begin{align*}d_{(P,Q)}\psi: T_{(P,Q)}\mathcal C \to T_{(Q,t)} \mathcal H, &&
    (U, V) \mapsto (V, \mathrm{trace}(PV))
    \end{align*}
    is surjective.
    Thus $T_{(Q,t)}\mathcal H = 
        \{(V, \mathrm{trace}(PV)) : V \in T_{Q}\mathrm{pGr}(2,n)\}$ 
        for any $P$ satisfying 
    \begin{align}\label{eq:Pcond}
         (P, Q) \in \mathcal C, && \mathrm{rank}(d_{(P,Q)}\psi) = d && \textrm{and } && F(P,Q) = t.
    \end{align}

    Thus the maps $\ell^P_{Q}: T_{Q}\mathrm{pGr}(2,n) \to \mathbb C$ defined by $V \mapsto \mathrm{trace}(PV)$ 
    are equal for all $P$ satisfying \eqref{eq:Pcond}. 
    We write $\ell_Q$ for this map.

    Fix generic $(Q,t)\in \mathcal H$. 
    Since $\mathrm{O}(n)$ acts transitively on the projection Grassmannian, 
     we can, without loss of generality, assume that $Q =\mathrm{Id}_2 \oplus 0_{n-2}$. 
    Since $P \notin {\rm cut}(Q)$, ${\rm im}(P)\cap {\rm ker}(Q)=0$ and there exists a unique matrix $A \in \mathbb C^{(n-2)\times 2}$ such that ${\rm im}(P)$ is the graph of $A$: ${\rm im}(P)=\big \{\begin{bmatrix}
        u \\ Au
    \end{bmatrix} : u \in \mathbb C^2\big \} = \mathrm{im}(Z_A)$ where $Z_A\!:=\!\begin{bmatrix}
       \mathrm{Id}_2 \\ A
    \end{bmatrix} \!\!\in\! \mathbb C^{n\times 2}$; see \cite[p.58]{MS74}.  
    We show that the matrix $Z_A^\top Z_A \!=\! \mathrm{Id}_2+A^\top A \in \mathbb C^{2\times 2}$ is invertible.
    Let $c \in \mathbb C^2$ be such that 
    $Z_A^\top(Z_A c)\!=\!0$.
    Then $Z_A c \in {\rm ker}(P)\cap {\rm im}(P)= 0$.
    By construction $(Z_A c)^\top = \begin{bmatrix}
        c^\top & (A c)^\top
    \end{bmatrix}$, hence $c=0$ and~$\mathrm{Id}_2+A^\top A$ is invertible. 
    Setting $S:=(\mathrm{Id}_2+A^\top A)^{-1}$, we obtain the expression
    $P= Z_A(Z_A^\top Z_A)^{-1}Z_A^\top = \begin{bmatrix}
        S & SA^\top \\
        AS & ASA^\top
    \end{bmatrix}.$
    Let $T_Q{\rm pGr}(2,n)=\bigg \{\begin{bmatrix}
        0 & H^\top \\ H & 0
    \end{bmatrix} \ | \ H \in \mathbb C^{(n-2)\times 2}\bigg \}$ be the tangent space and $Q_H \in T_Q{\rm pGr}(2,n)$. 
    Then $\ell_Q(Q_H) = {\rm trace}(PQ_H) =2{\rm trace}(H^\top B)$ for $B:=AS$. Since $\ell_Q$ does not depend on $P$ and $\mathrm{trace}(-,-)$ is non-degenerate, $(Q,t)$ completely determines $B$.
    The relation $S - S^2 = B^\top B$, which holds by construction, has infinitely many solutions. Indeed, 
    every solution $S$ uniquely determines a point $P$ in the fiber via $A=BS^{-1}$.

    We now prove that if $(Q,t)$ has a positive dimensional fiber, i.e., is generic in the sense above, then $t = F(P,Q)={\rm trace}(PQ)={\rm trace}(S) = 1$.
    Let 
    $\delta = {\rm det}(S)$. By Cayley-Hamilton $S^2-tS+\delta \mathrm{Id}_2=0$, and thus, $B^\top B = (1-t)S +\delta \mathrm{Id}_2$. 
    Assume for contradiction that $t \neq 1$. 
    Then $S=\frac{1}{1-t}(B^\top B-\delta \mathrm{Id}_2)$ and $\delta = ({\rm trace}(B^\top B)-t+t^2)/2$.
    Since $B$ is fixed for $(Q,t)$, the matrix $S$ is determined by $t$.
    Thus the fiber $(Q,t)$ has only a single point, a contradiction.
    We conclude that if $(Q,t)$ has a positive dimensional fiber, then $t = 1$.
    It follows that $F(P,Q) = 1$ on a dense open subset of $\mathcal C$, and since $F$ is continuous, $F$ is constant on $\mathcal C$.

    As above, let $Q = \mathrm{Id}_2 \oplus 0_{n-2}$ and let $P$ be such that $(P,Q)\in \mathcal C$.
    Since $F$ is constant on $\mathcal C$,
    ${\rm trace}(PQ_H)=0$ for every $Q_H\in T_Q{\rm pGr}(2,n)$ which implies $B=0$ and $S^2-S=0$. 
    Since $S$ is invertible, it must be $S=\mathrm{Id}_2$ and $P=Q$, in which case $F(P,Q)=2$, a contradiction. Hence, $\mathrm{dim}(\mathcal C) = d$ and the fibers of $\pi_2$ are generically finite. 
\end{proof}

Low dimensional models may have no extraneous critical points, i.e., $\mathcal E_{\mathcal M}^Q = \emptyset$. If, in addition, the model $\mathcal M$ is GD-general, then Definition~\ref{def:GD} coincides with the definition of the ED degree. 
Thus, the set $\mathcal E_{\mathcal M}^Q$ tells us whether the ED and GD degrees of $\mathcal M$ are equal. 

\begin{theorem}\label{thm:EDGD}
    The ED and GD degrees of a GD-general $\mathcal M$ agree if and only if $\mathcal E_{\mathcal M}^Q$ is empty for generic $Q$. 
    The set $\mathcal E_{\mathcal M}^Q$ is expected to be empty if 
    \begin{align*}
        \dim \mathcal M <
        \begin{cases}
        k(n-k) - \lfloor n^2/8 \rfloor &\textrm{if } n \leq 4k\\
        k^2 &\textrm{if } n \geq 4k.
        \end{cases}
    \end{align*}
\end{theorem}
\begin{proof}
    The first claim follows from the definition of GD degree;
    the second is an immediate consequence of the following result about the dimension of the critical point set of~\eqref{eq:proj-optprob}.
\end{proof}

We will describe the critical points of \eqref{eq:proj-optprob} using Schubert varieties.
Fix $Q \in \mathrm{pGr}(k,n)$. 
For a subspace $W \subset \CC^n$ and $j \leq \min\{{\rm dim}(W),k\}$, we denote by $\Omega(W,j)$ the Schubert variety $\Omega(W,j) := \{P \in {\rm pGr}(k,n) \ | \ {\rm dim}(\mathrm{im}\,P\cap W)\geq j\}$.
For $Q \in \mathrm{pGr}(k,n)$, we write
\[\Omega_i(Q)  :=  \Omega(\mathrm{ker}\,Q,i) \,\cap \,\Omega(\mathrm{im}\,Q,k-i).\]

\begin{lemma}
    The critical point set of \eqref{eq:proj-optprob} is the disjoint union $\Omega(Q) = \bigsqcup_{i=0}^k \Omega_i(Q)$. 
    The codimension of $\Omega_i(Q)$ is $k(n-k) - i(n-2i)$; the codimension of $\Omega(Q)$ is $k(n-k)-\lfloor \frac{n^2}{8}\rfloor$ if $n\leq 4k$, and $k^2$ if $n\geq 4k$.
\end{lemma}
\begin{proof}
    If $P \in {\rm pGr}(k,n)$ is critical for \eqref{eq:proj-optprob}, then, by Proposition \ref{prop:cps},
      \[\mathrm{im}\,P = (\mathrm{im}\,P \cap \mathrm{im}\,Q) \oplus (\mathrm{im}\,P \cap \mathrm{ker}\,Q).\]  
    The description of the critical points in the statement accounts for all different possible dimensions of the two intersections on the righthand side.
    Since $\mathrm{im}\,Q \cap \mathrm{ker}\,Q =0$, the intersection $\Omega(\mathrm{ker}\,Q,i) \cap \Omega(\mathrm{im}\,Q,k-i)$ is transverse.
    Counting boxes in a Young diagram, we find that if ${\rm dim}(W)=\ell$, then ${\rm codim}(\Omega(W,j)) = j(n-k+j-\ell)$ in ${\rm pGr}(k,n)$.
    Thus ${\rm codim}(\Omega_i(Q)) = {\rm codim}(\Omega(\mathrm{ker}\,Q,i)) + {\rm codim}(\Omega(\mathrm{im}\,Q,k-i)) = k(n-k) -i(n - 2i)$. The codimension of $\Omega(Q)$ is obtained by taking the minimum of this value over $i= 0, \ldots, k$.
\end{proof}

The slogan of Theorem~\ref{thm:EDGD} is that the ED and GD degrees agree for most low-dimensional models and disagree for most high-dimensional models.

\begin{example}\label{ex:twisted_cubic}
    The Chow variety $\mathcal M$ of the twisted cubic curve is cut out in $\mathrm{Gr}(2,4) \subseteq \mathbb P^5$ by the Bezout determinant: 
    \begin{equation*}
        \begin{vmatrix}
            x_{12} & x_{13} & x_{23}\\
            x_{13} & x_{14} + x_{23} & x_{24}\\
            x_{23} & x_{24} & x_{34}
        \end{vmatrix}=0.
    \end{equation*}
    This threefold has ED degree 42 and GD degree 10 by \cite[Section 5]{FRS}. 
    The cut locus of a point in $\mathrm{Gr}(2,4)$ is also a threefold and it is expected to intersect $\mathcal M$ in a surface. 
    We visualize this in Figure~\ref{fig:twisted_cubic} by slicing the affine open of $\mathrm{Gr}(2,4) \subseteq \mathbb P^5$ where $x_{12} + x_{13} + x_{14} + x_{23} + x_{24} + x_{34} = 1$ with a $3$-dimensional affine linear space. 
    The Chow variety $\mathcal M$ is shown in orange and the cut locus of the magenta data point is shown in red. 
    Two of the 10 critical points of \eqref{eq:proj-optprob} are shown in green on the orange curve, one in each panel. 
    Finally, the yellow cross marks one of the extraneous critical points in the intersection of the cut locus.
\end{example}
\begin{figure}
    \centering
    \includegraphics[width=0.45\linewidth]{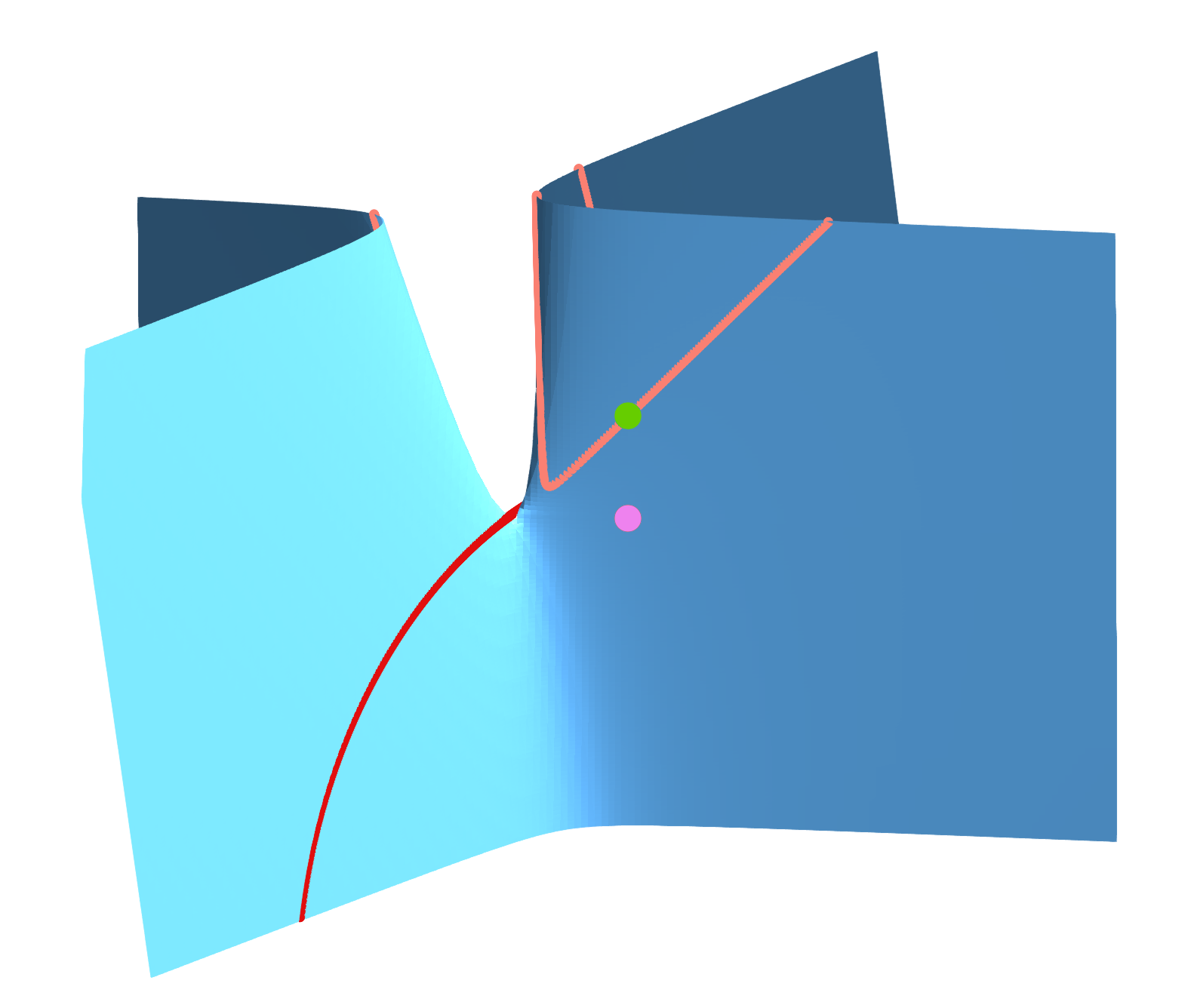}
    \includegraphics[width=0.45\linewidth]{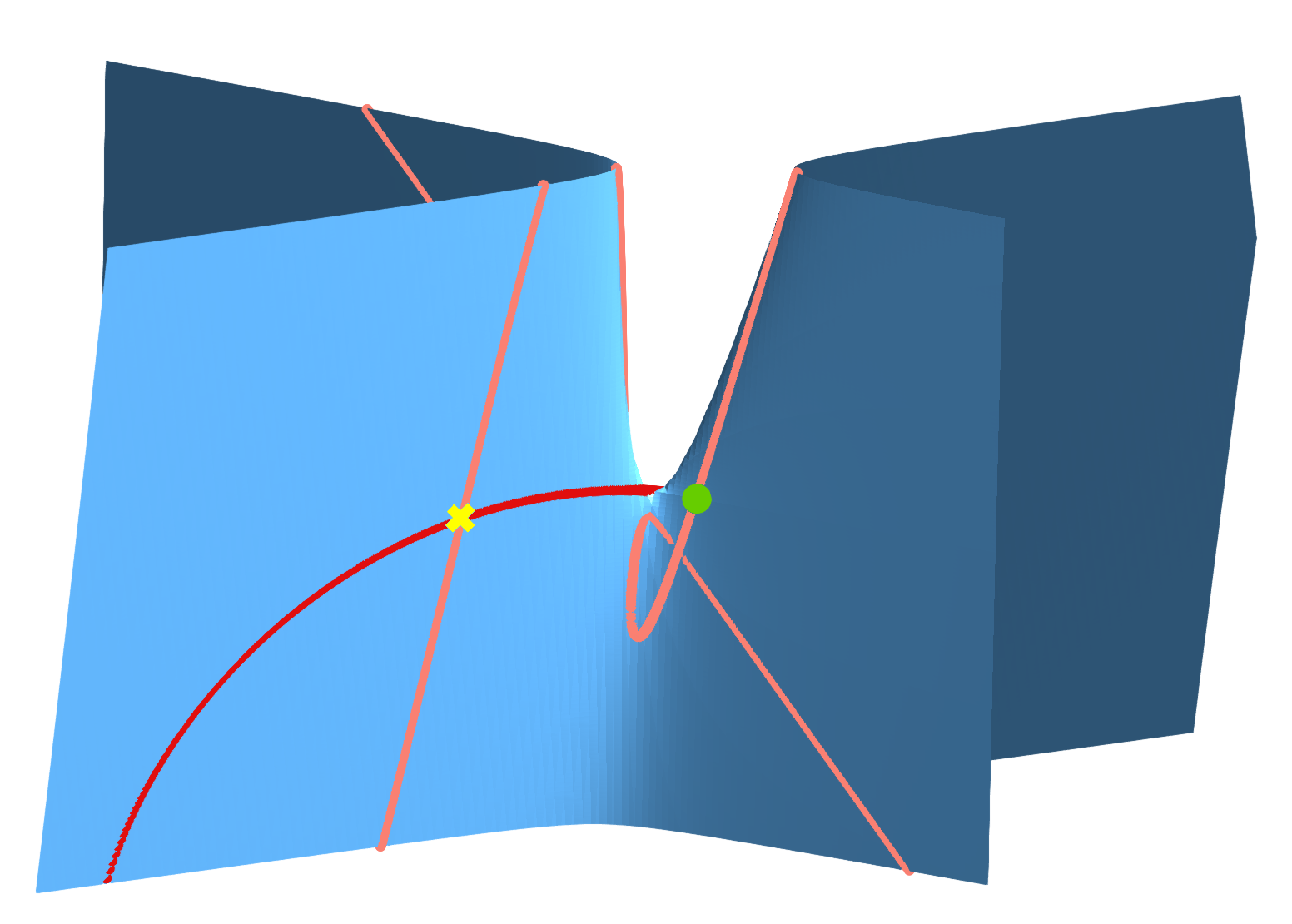}
    \caption{The \textcolor{wine}{cut locus} of a \textcolor{mag}{data point} along with true \textcolor{gre}{critical points} and an \textcolor{yel}{extraneous critical point} for \eqref{eq:trace-opt} on the \textcolor{salmon}{Chow variety of the twisted cubic} in Example~\ref{ex:twisted_cubic}.}
    \label{fig:twisted_cubic}
\end{figure}

\section{Products of Grassmannians}\label{sec:products}
Now we turn to the computation of the ED and GD degrees of the product of two smaller Grassmannians embedded in $\mathrm{pGr}(k,n)$. 
Our first motivation for studying this model is that, in Sections \ref{sec:positroids} and \ref{sec:schubert}, we will see that certain matroid and Schubert varieties are products of two smaller Grassmannians. The ED degree of this model is straightfoward to compute, but we will see that the GD degree is already nontrivial.

Our main goal is to prove Theorem 
\ref{thm:prod_Grassmannian}, and we will characterize the actual critical points that give rise to the respective ED and GD degrees. 
Our first results address the fundamental linear algebra of submatrices of points in the projection Grassmannian. 
For instance, the next result shows that any sufficiently large principal submatrix of a matrix in the projection Grassmannian has $1$ as an eigenvalue. 
\begin{lemma}\label{lemma: eigenvalue1}
    Each principal $r \times r$ submatrix of $Q \in \mathrm{pGr}(k,n)$ has an eigenvalue $1$ whose geometric multiplicity is at least $k + r - n$. Equality holds for generic $Q$. 
\end{lemma}
\begin{proof}
    Without loss of generality, we consider $Q'$, the leading $r \times r$ principal submatrix of $Q$. Let $E' = \mathrm{span}\{e_1, \ldots, e_r\}$.  Then ${\rm dim}({\rm im}(Q)\cap E')\geq k+r-n$, with equality for generic $Q$. For $x = \begin{bmatrix} x' \\ 0 \end{bmatrix}\in {\rm im}(Q)\cap E'$ we have $Q' x' = x'$, and the statement follows. 
\end{proof}

Any $n \times n$ symmetric matrix may be written as
$Q = \begin{bmatrix} Q_{11} & Q_{12} \\
Q_{12}^\top & Q_{22}
\end{bmatrix}.$
If $Q \in \mathrm{pGr}(k,n)$, there is an important relationship between the blocks $Q_{11}, Q_{12}, Q_{22}$ of this matrix. 
The next result, inspired by the CS decomposition of an orthogonal matrix \cite{PW94}, shows that the left and right singular vectors of $Q_{12}$ are eigenvectors of $Q_{11}$ and $Q_{22}$, respectively.

\begin{proposition}\label{prop:SVD}
Assume $Q \in \mathrm{pGr}(k,n)$ with $Q_{11} \in \mathrm{Sym}^2 \mathbb C^{m}$, $Q_{22} \in \mathrm{Sym}^2 \mathbb C^{n-m}$, and $Q_{12} \in \mathbb C^{m \times (n-m)}$, as above, for $k\leq m, n/2$.
For generic $Q$, the following holds:
\begin{enumerate} 
\item If $k \leq n-m$, there exist $U \in \mathrm{O}(m)$ and $V \in \mathrm{O}(n-m)$ and diagonal $k \times k$ matrices $C, S$ with $C^2 + S^2 = \mathrm{Id}_k$ such that 
\begin{align*}
    Q = 
     \begin{bmatrix}
            U^\top & 0 \\ 0 & V^\top
        \end{bmatrix}
        \begin{bNiceArray}[margin]{cc|cc}
            C^2 & 0 &  CS & 0\\ 
            0 & 0 & 0 & 0\\
            \hline 
            CS  & 0 & S^2  & 0\\
            0 & 0 & 0 & 0
        \end{bNiceArray}
        \begin{bmatrix}
            U & 0\\ 0 & V
        \end{bmatrix}.   
\end{align*}
\item 
If $k > n-m $,  there exist  $U \in \mathrm{O}(m)$ and $V \in \mathrm{O}(n-m)$ and diagonal $(n-m) \times (n-m)$ matrices $C, S$ with $C^2 + S^2 = \mathrm{Id}_{n-m}$ such that 
\begin{align*}
    Q = 
     \begin{bmatrix}
            U^\top & 0\\ 0 & V^\top
        \end{bmatrix}
        \begin{bNiceArray}[margin]{ccc|c}
            C^2 & 0 & 0 &  CS \\
            0 & {\rm Id}_{k-(n-m)} & 0& 0 \\
            0 & 0 & 0 & 0\\
            \hline 
            CS  & 0 & 0 & S^2 \\
        \end{bNiceArray}
        \begin{bmatrix}
            U & 0 \\  0 & V
        \end{bmatrix}.   
\end{align*}
\end{enumerate}
\end{proposition}
\begin{proof}
 We prove the case $k \leq n-m$. 
   The case $k > n-m$ follows by a nearly identical argument combined with an application of Lemma~\ref{lemma: eigenvalue1} above.
    
     For $k \times k$ diagonal matrices $C$ and $S$ fix the notation
    \begin{align*}
        \widehat{C} = \begin{bmatrix}
        C & 0 \\ 0 & 0
        \end{bmatrix} &&
        \widehat{S} = \begin{bmatrix}
        S & 0 \\ 0 & 0
        \end{bmatrix} &&
        \widehat{CS} = 
         \begin{bmatrix}
        CS & 0 \\ 0 & 0
        \end{bmatrix} 
    \end{align*}
    so that $\widehat{C}$ is $m \times m$   and $\widehat{S}$ is  $(n-m) \times (n-m)$ , 
    and $\widehat{CS}$ is $m \times (n - m)$. 
   Since $Q$ is generic, $Q_{11}$ is diagonalizable and thus orthogonally diagonalizable \cite[Theorem XI.4]{gantmacher}. So there exists $U \in \mathrm{O}(m)$ and a complex diagonal matrix $\widehat{C}^2$ such that $Q_{11} = U^\top \widehat{C}^2U$. Here, $C^2$ (and therefore $C$) is a $k\times k$ diagonal matrix with nonzero diagonal entries since the rank of $Q_{11}$ is 
   generically equal to $k$. From $Q^2 = Q$, we have $Q_{12}Q_{12}^\top = Q_{11} - Q_{11}^2 = U^\top(\widehat{C}^2 -\widehat{C}^4)U$. 
    Setting $S^2 = \mathrm{Id}_k - C^2$,  
    we have $Q_{12}Q_{12}^\top 
    = U^\top \widehat{CS} \widehat{CS}^\top U =   (U^\top \widehat{CS})(U^\top \widehat{CS})^\top$.
    Again, the genericity of $Q$ implies that $Q_{12}Q_{12}^\top$ and $CS$ have rank $k$. 
    By Lemma~\ref{lem:tech} below, there exists $V \in \mathrm{O}(n-m)$, such that $Q_{12} = U^\top \widehat{CS} V$. 
    Now using the equation $ Q_{12} = Q_{11}Q_{12} + Q_{12}Q_{22}$ and the computations above we obtain $\widehat{CS} VQ_{22}V^\top = \widehat{CS} - \widehat{C}^2\widehat{CS} =  \widehat{CS} \widehat{S}^2$. 
    By our genericity assumption, the diagonal entries of ${CS}$ are nonzero, so the upper left $k \times k$ block of $VQ_{22}V^\top$ is $S^2$, and the upper right and lower left blocks are zero. 
    Since the rank of $Q_{22}$ cannot exceed $k$, the lower right block must be zero as well. 
\end{proof}

\begin{lemma}\label{lem:tech}
    Let $A, B \in \CC^{p \times q}$ have rank $r$. 
    If $AA^\top = BB^\top$ and $AA^\top$ has rank $r$, then there exists $V \in \mathrm{O}(q)$ such that $A = BV$.
\end{lemma}
\begin{proof}
    Since $\mathrm{ker}(A^\top) \subseteq \mathrm{ker}(AA^\top)$ and both matrices have the same rank, the kernels are equal.
    Similarly $\mathrm{ker}(B^\top) = \mathrm{ker}(BB^\top)$, so we conclude $\mathrm{ker}(A^\top) = \mathrm{ker}(B^\top)$. 
    We now define an isometry $T$ between $\mathrm{im}(A^\top)$ and $\mathrm{im}(B^\top)$. 
    Given a vector $v \in \mathrm{im}(B^\top)$, there exists a vector $x \in \CC^p$ such that $v = B^\top x$. 
    Define $T(v) = A^\top x$.
    This map is well-defined since $\ker(A^\top) = \ker(B^\top)$. 
    To see that $T$ is an isometry, use $AA^\top = BB^\top$: given $x,y \in \CC^p$, 
    \begin{align*}
        (B^\top x)^\top (B^\top y) = x^\top B B^\top y = x^\top AA^\top y = (A^\top x)^\top (A^\top y).
    \end{align*}
    Witt's theorem guarantees that $T$ may be extended to an isometry on $\CC^q$. 
    Such an isometry is represented by a matrix in $V^\top \in \mathrm{O}(q)$ such that $V^\top B^\top = A^\top$ and the result follows. 
\end{proof}

We now return to our Euclidean distance problem. 
Before examining products of Grassmannians, we first turn to a single Grassmannian $\mathrm{pGr}(k,r)$ with $r \leq n$  embedded in $\mathrm{pGr}(k,n)$ as the collection of projection matrices $\begin{bmatrix} P' & 0 \\ 0 & 0\end{bmatrix}$ with $P' \in \mathrm{pGr}(k,r)$. 

\begin{corollary}\label{cor:small-grass}
    The ED and GD degrees of $\mathrm{pGr}(k,r) \subseteq \mathrm{pGr}(k,n)$ are $\binom{r}{k}$ and $1$, respectively.
    For generic data $Q \in \mathrm{pGr}(k,n)$, the unique $GD$-critical point is the symmetric projection onto the image of the upper left $r \times r$ block of $Q$.
\end{corollary}
\begin{proof}
    For the ED degree, the problem reduces to \eqref{eq:proj-optprob} and thus follows from Theorem~\ref{thm:disc1}.
    Since the problem \eqref{eq:trace-opt} is $\mathrm{O}(r) \times \mathrm{O}(n - r)$-equivariant for this model, we can assume that $Q \in \mathrm{pGr}(k,n)$ has the canonical block form in Proposition~\ref{prop:SVD}. Since $Q$ is generic, the upper left $r \times r$ block of $Q$, denoted $Q'$, is a diagonal matrix with $k$ nonzero diagonal entries. The optimization problem \eqref{eq:trace-opt} reduces to 
    $\max_{P' \in \mathrm{pGr}(k,r)} \mathrm{trace}(P'Q')$. By Proposition~\ref{prop:cps}, the point $P' = \mathrm{Id}_k \oplus 0_{n-k}$ is critical.  
    In this case, $\mathrm{im}(P') = \mathrm{im}(Q')$. 
    For every other critical point $P'$, there exists a nonzero vector $v \in \mathrm{im}(P') \cap \mathrm{ker}(Q')$, and  hence $\begin{bmatrix} v^\top & 0 \end{bmatrix}^\top \in \mathrm{im}(P) \cap \mathrm{ker}(Q)$, implying that $P \in \mathrm{cut}(Q)$ is extraneous. 
    Thus $P' = \mathrm{Id}_k \oplus 0_{n-k}$ is the unique critical point.
\end{proof}

\vspace{0.3cm}
We now turn our attention to computing the ED and GD degrees of $\mathrm{pGr}(j,r) \times \mathrm{pGr}(\ell,s)$ embedded in $\mathrm{pGr}(k,n)$. In the following results, we will assume 
\begin{align*}
    2k  \leq n && j + \ell = k &&  s \leq r < r + s \leq n.
\end{align*}
We embed our product of Grassmannians in $\mathrm{pGr}(k,n)$ as
\begin{equation*} 
\mathrm{pGr}(j,r) \times \mathrm{pGr}(\ell,s) = \biggl\{ \begin{bmatrix}
            0_{n-(r+s)} & 0 & 0 \\
            0 & P_{22} & 0\\
            0 & 0 & P_{33}
        \end{bmatrix}  :  P_{22} \in \mathrm{pGr}(j,r) \mbox{ and } P_{33} \in \mathrm{pGr}(\ell,s) \biggr\}. 
\end{equation*}

Alternatively, one can view this model in Pl\"ucker coordinates.
To write the ideal, let
$R := \{n - r - s+1, \ldots,  n-s\}$ and $S := \{n-s+1, \ldots, n\}$ be the column indices of $P_{22}$  and $P_{33}$, respectively, and set $
    \mathcal C := \{J \in {\textstyle \binom{[n]}{k}} : |J \cap R| = j \textrm{ and } |J \cap S| = \ell\}.$
\begin{proposition}\label{prop:pluecker}
The prime ideal of $\mathrm{Gr}(j,r) \times \mathrm{Gr}(\ell,s) \subseteq \PP^{\binom{n}{k}-1}$ is 
    $     I_{\mathrm{Gr}(k,n)} + 
    \langle x_J : J \notin \mathcal  C
    \rangle$.
\end{proposition}
\begin{proof}
    Let $I = I_{\mathrm{Gr}(k,n)} + 
    \langle x_J : J \notin \mathcal  C
    \rangle$.
    We will prove that the following map is injective:
    \begin{align*}
         \varphi : \frac{\mathbb C[x_J: J \in {\textstyle \binom{[n]}{k}}]}{I} \longrightarrow
         \frac{\mathbb C[y_J: J \in {\textstyle \binom{R}{j}}]}
         {I_{\mathrm{Gr}(j,r)}}
         \otimes_\mathbb C 
         \frac{\mathbb C[z_J: J \in {\textstyle \binom{S}{\ell}}]}
         {I_{\mathrm{Gr}(\ell,s)}},
         &&
        x_J \mapsto y_{J \cap R} \otimes z_{J \cap S}
    \end{align*}
    where we set $y_{J\cap R}, z_{J\cap S} = 0$ if $J \notin \mathcal C$.
    The map $\varphi$ is a homomorphism of graded rings with $\deg(x_J) = k$, $\deg(y_J) = j$ and $\deg(z_J) = \ell$. 
    It therefore suffices to check injectivity of the linear maps on the vector spaces.
    By \cite[Section 5.4]{LB}, the set
    $\mathrm{Std}_{\alpha k}(\mathrm{Gr}(k,n)) = 
    \{x_{J_1}\cdots x_{J_\alpha} : |J_i|=k,\,  J_1 \leq \cdots \leq J_{\alpha}\}
    $
    where $J \leq J'$ if $j_i \leq j'_i$ for $i \in [k]$
     is a basis for the vector space $(\mathbb C[x]/I_{\mathrm{Gr}(k,n)})_{\alpha k}$.
     Therefore 
    \begin{equation}\label{eq:std-mon}
        \{x_{J_1}\cdots x_{J_\alpha} \in \mathrm{Std}_{\alpha k}(\mathrm{Gr}(k,n)): J_i \in \mathcal C \textrm{ for all } i = 1, \ldots, \alpha
        \}
    \end{equation}
    is a basis for the vector space $(\mathbb C[x]/I)_{\alpha k}$.
    It suffices to show that the image of \eqref{eq:std-mon} under $\varphi$ is linearly independent.
    Indeed, the image of a standard monomial is the product of two standard monomials: 
    \begin{equation*}
        \varphi(x_{J_1} \cdots x_{J_\alpha}) = 
        \prod_{i = 1}^\alpha y_{J_i \cap R} \otimes \prod_{i = 1}^\alpha z_{J_i \cap S}
    \end{equation*}
    The tensor products of basis elements of a vector space form a linearly independent set, so the image of the standard monomial basis for $(\mathbb C[x]/I)_{\alpha k}$ is linearly independent.
    Thus $\varphi$ is injective, and since the righthand side is a domain, $I$ is prime.
\end{proof}

\noindent
For the remainder of this section, we partition the data matrix $Q$ to match the partition~of~$P$:
\begin{align*}
    Q = 
    \begin{bmatrix}
        Q_{11} & Q_{12} & Q_{13}\\
        Q_{12}^\top & Q_{22} & Q_{23}\\
        Q_{13}^\top & Q_{23}^\top & Q_{33}
    \end{bmatrix}.
\end{align*}
Here $Q_{11} \in\mathrm{Sym}^2 \mathbb C^{n - r - s}$, $Q_{22} \in\mathrm{Sym}^2 \mathbb C^{r}$, and $Q_{33} \in\mathrm{Sym}^2 \mathbb C^{s}$. 
We note that when $r+s = n$, the partition above reduces to $
    Q = 
    \begin{bmatrix}
        Q_{22} & Q_{23}\\
         Q_{23}^\top & Q_{33}
    \end{bmatrix}.$
\begin{lemma} \label{lem:ED-critpoints-of-products-of-Grassmannians}
 The ED critical points for  $\mathrm{pGr}(j,r) \times \mathrm{pGr}(\ell,s) \subseteq \mathrm{pGr}(k,n)$ are pairs of projections $(P_{22}, P_{33})$ onto $j$-dimensional $Q_{22}$- and $\ell$-dimensional $Q_{33}$-invariant subspaces.    
\end{lemma}  
\begin{proof}
The Euclidean distance optimization problem  \eqref{eq:trace-opt} reduces to 
     \begin{align*}
       \textrm{maximize} \quad \mathrm{trace}(P_{22}Q_{22}) + \mathrm{trace}(P_{33}Q_{33}) \quad \textrm{subject to} \quad (P_{22},P_{33}) \in \mathrm{pGr}(j,r) \times \mathrm{pGr}(\ell,s),
    \end{align*}
    and Proposition \ref{prop:cps} implies the result. 
\end{proof}
We now turn to the GD degree. 
The critical points are still invariant subspaces as above, but some of them are now in $\mathrm{cut}(Q)$ and have to be removed. 
We assume in the next results that $Q$ is real. 
Since the real points of the projection Grassmannian are Zariski dense, the GD degree is the number of complex critical points for generic real data.
In particular, this assumption implies that $Q_{22}$ and $Q_{33}$ are diagonalizable.

  In the rest of this section, for any $x \in \mathbb Z$ let $x_+ := \max\{x,0\}$. 
  \begin{lemma}\label{lem:eigenspaces}
      For generic $Q \in \mathrm{pGr}(k,n)$, $Q_{22}$ has eigenvalue $1$ with multiplicity $(k + r - n)_+$, along with $\mathrm{min}\{k,r,n-r\}$ distinct eigenvalues in $(0,1)$, and an $(r-k)_+$-dimensional kernel;
      $Q_{33}$ has $\min\{k,s\}$ distinct eigenvalues in $(0,1)$ and an $(s - k)_+$-dimensional kernel.
  \end{lemma}
  \begin{proof}
  We have ${\rm rank}(Q_{22})=\min\{k,r\}$, hence $Q_{22}$ has an $(r-k)_+$-dimensional kernel. By Lemma \ref{lemma: eigenvalue1}, $Q_{22}$ has an eigenvalue $1$ of multiplicity $(k+r-n)_+$. 
  Generically, the remaining $\min\{k,r,n-r\}$ eigenvalues are distinct  and they are in $(0,1)$ by Cauchy's interlacing theorem. Similarly, $Q_{33}$ has an $(s-k)_+$-dimensional kernel, but generically no eigenvalue equal to $1$, as $k+s-n\leq 0$. The rest of the  $\min\{k,s\}$ eigenvalues are distinct in $(0,1)$.
  \end{proof}

Lemma~\ref{lem:eigenspaces} gives a complete picture of the eigenspaces of $Q_{22}$ and $Q_{33}$; we now give them names.
Let $f_1, \ldots, f_{(k+r-n)_+} \in \mathbb R^r$ be a basis for the eigenspace of $Q_{22}$ with eigenvalue $1$ and set $\hat f_i = (0,f_i,0) \in \RR^n$.
Let $u_i \in \RR^r$ be such that $Q_{22}u_i = \lambda_i u_i$ with $\lambda_i \in (0,1)$ and $\hat{u}_i = (0,u_i,0)\in \RR^n$ for $i=1,\dots,\min\{k,r,n-r\}$. Similarly,  let $v_h \in \RR^s$ be such that $Q_{33}v_h = \mu_h v_h$ for $\mu_h \in (0,1)$ and $\hat{v}_h = (0,0,v_h)\in \RR^n$ for $h=1,\dots,\min\{k,s\}$.
\[
\begin{array}{ccccc}
\toprule
     & \textrm{Eigenvalue} & \textrm{Eigenspace} & \textrm{Basis} \\
     \midrule 
    Q_{22} & 1 & F & \{f_1, \ldots, f_{(k+r-n)_+}\}\\
    Q_{22} & 0< \lambda_i < 1 & U_i & \{u_i\} & \textrm{for } i = 1, \ldots, \mathrm{min}\{k,r,n-r\}\\
    Q_{33} & 0< \mu_h < 1 & V_h & \{v_h\} & \textrm{for } i = 1, \ldots, \mathrm{min}\{k,s\}\\
    \bottomrule
\end{array}
\]
Henceforth, $I \subseteq \{1, \ldots, \mathrm{min}\{k,r,n-r\}\}$ and $H \subseteq \{i = 1, \ldots, \mathrm{min}\{k,s\}\}$ will denote sets of sizes $|I| = \delta$ and $|H| = \ell$ where $\delta := j-(k+r-n)_+$.
  
\begin{lemma} \label{lem:cut-locus-for-products}  
If 
$(P_{22}, P_{33}) \in  \mathrm{pGr}(j,r) \times \mathrm{pGr}(\ell,s)$ is a GD critical point then  $P_{22}$ projects onto a subspace of the form $F \oplus \bigoplus_{i \in I}U_i$ and $P_{33}$ projects onto a subspace of the form $\bigoplus_{h \in H}V_h$.
\end{lemma}
\begin{proof}
The fact that the critical points are invariant spaces of the appropriate dimensions follows from Proposition \ref{prop:cps}.
We now describe which points to remove.
Let $u \in \RR^r$ such that $Q_{22}u = \lambda u$ and let $\hat{u}:=(0,u,0)\in \RR^n$. Then 
  \begin{align}\label{eq:eigennorms}
      \|Q\hat{u}\|^2 = \hat{u}^\top Q \hat{u} = u^\top Q_{22}u = \lambda\|u\|^2 \,.
  \end{align}
  Similarly, if $v \in \RR^s$ is such that $Q_{33}v=\mu v$ and $\hat{v}=(0,0,v)$, then $\|Q\hat{v}\|^2=\mu \|v\|^2$. 

    Assume $(P_{22}, P_{33})$ is critical and let $u \in \ker(Q_{22}) \cap \mathrm{im}(P_{22})$ be nonzero. Then $P\hat{u} = \hat{u}$, so that $\hat{u} \in \mathrm{im}(P)$.
  By \eqref{eq:eigennorms}, $\|Q\hat u\|=0$ and $\hat u \in {\rm ker}(Q)\cap {\rm im}(P)$. This is equivalent to $P \in {\rm cut}(Q)$. The same conclusion holds if $v \in \ker(Q_{33}) \cap \mathrm{im}(P_{33})$.

We now prove that $F \subseteq \mathrm{im}(P_{22})$. 
Since $P_{22}Q_{22} = Q_{22}P_{22}$, the vector space $F$ decomposes as $F = (\mathrm{im}(P_{22}) \cap F) \oplus (\mathrm{ker}(P_{22}) \cap F)$.
    Suppose for contradiction that $u \in \mathrm{ker}(P_{22}) \cap F$ is nonzero. 
    Then $\hat u \in \mathrm{ker}(P)$
  and by \eqref{eq:eigennorms}
  \begin{align*}
      \|({\rm Id}_n-Q)\hat{u}\|^2 = \hat{u}^\top({\rm Id}_n-Q)\hat{u} = \|\hat{u}\|^2 - \hat{u}^\top Q\hat u = 0\,.
  \end{align*}
    Thus,
$\hat{u} \in \ker({\rm Id}_n-Q) =\mathrm{im}(Q)$. 
It follows that 
$\ker(P) \cap \mathrm{im}(Q) \neq 0$. Since $P$ and $Q$ span subspaces of the same dimension, this condition is equivalent to $\mathrm{im}(P)\cap \ker(Q)\neq 0$, that is $P \in {\rm cut}(Q)$. This contradiction shows that $F \cap \ker(P_{22})$ is trivial and $F \subseteq \mathrm{im}(P_{22})$. 
\end{proof}

This lemma gives a finite set of possible critical points. 
We will now show that if $n > r + s$, then all of the points in Lemma~\ref{lem:cut-locus-for-products} are critical, and that if $n = r+s$, some must be removed.
 
 Assume $P_{22}$ projects onto $F \oplus \bigoplus_{i \in I}U_i$ 
and $P_{33}$ projects onto  $\bigoplus_{h \in H}V_h$.
 Then $P \notin {\rm cut}(Q)$ if and only if $Q$ is injective on ${\rm im}(P)$ as in \eqref{eq:cut}. 
  Equivalently, the vectors $Q\hat{f}_1,\dots,Q\hat{f}_{(k+r-n)_+}$, $\{Q\hat{u}_i\}_{i \in I}$ and $\{Q\hat{v}_h\}_{h \in H}$ are linearly independent. Fix the notation
  \begin{align*}
      x_i := \frac{Q \hat u_i}{\sqrt{\lambda_i}}\,, \quad y_h:=\frac{Q \hat v_h}{\sqrt{\mu_i}}\,.
  \end{align*}
  As in the proof of Lemma \ref{lemma: eigenvalue1}, $Q\hat{f}_i = \hat{f}_i$ for all $i \in[(k+r-n)_+]$, so $P \notin {\rm cut}(Q)$ if and~only~if
  \begin{align} \label{eq:linind1}
      \hat{f}_1,\dots,\hat{f}_{(k+r-n)_+}, \{x_i\}_{i \in I}, \{y_h\}_{h \in H} \ \textrm{ are linearly independent in $\RR^n$.}
  \end{align}
  We factor $Q=ZZ^\top$ where the columns of  $Z \in \RR^{n\times k}$ 
  form an orthonormal basis of $\mathrm{im}(Q)$, i.e., 
  $Z^\top Z={\rm Id}_k$. Partition $Z$ as 
  \begin{align*}
      Z=\begin{bmatrix}
          Z_1 \\
          Z_2 \\
          Z_3
      \end{bmatrix}, \quad Z_1\in \RR^{(n-r-s) \times k}, \, Z_2 \in \RR^{r \times k}\,, Z_3 \in \RR^{s \times k}.
  \end{align*}
  In particular, $Z_2Z_2^\top=Q_{22}$ and $Z_3Z_3^\top=Q_{33}$, while $Z_1^\top Z_1+Z_2^\top Z_2+Z_3^\top Z_3={\rm Id}_k$. 
  We define 
  \begin{align*}
      z_i := \frac{Z_2^\top u_i}{\sqrt{\lambda_i}}\, , \quad w_h := \frac{Z_3^\top v_h}{\sqrt{\mu_h}}\,.
  \end{align*}
  By definition, $Zz_i = x_i$ and $Zw_h=y_h$. Moreover, since $Q_{22}u_i = \lambda_i u_i$, we obtain $Z_2^\top Z_2z_i = \lambda_i z_i$, and similarly, $Z_3^\top Z_3w_h = \mu_hw_h$. Since $\hat{f}_i \in {\rm im}(Q) = {\rm im}(Z)$, there exists unique $g_i \in \RR^k$ such that $\hat{f}_i = Z g_i$. From $Z^\top Z = {\rm Id}_k$, it follows that $Z:\RR^k \rightarrow {\rm im}(Q)$ is an isometry and \eqref{eq:linind1} is equivalent to 
  \begin{align}\label{eq:linind2}
      g_1,\dots,g_{(k+r-n)_+}, \{z_i\}_{i \in I}, \{w_h\}_{h \in H} \ \textrm{ are linearly independent in $\mathbb R^k$}\,. 
  \end{align}

\begin{lemma} \label{lem:n>r+s}
Suppose $n > r + s$.
If $P_{22}$ projects onto $F \oplus  \bigoplus_{u \in I}U_i$
and $P_{33}$ projects onto $\bigoplus_{h \in H}V_h$, then $(P_{22}, P_{33})$ is GD critical for $Q$.
\end{lemma}
\begin{proof} 
 For generic $Q$, the eigenvectors $\{z_i\}_{i \in I}$ of $Z_2^\top Z_2$ and the eigenvectors $\{w_h\}_{h \in H}$ of $Z_3^\top Z_3$ are algebraic functions of $Q$, and so is the determinant of the matrix obtained from the vectors in \eqref{eq:linind2}. We claim that for every choice of $I$ with $|I|=\delta$ and for every choice of $H$ with $|H| = \ell$ this determinant is not identically zero. It follows that \eqref{eq:linind2} can only fail on a proper Zariski closed subset of ${\rm pGr}(k,n)$. 
 Therefore, for generic $Q \in {\rm pGr}(k,n)$ \eqref{eq:linind2} holds for every $I$ and $H$ and every corresponding $P$ lies outside of ${\rm cut}(Q)$.

 To prove the claim, we construct an explicit $Q$ for which \eqref{eq:linind2} holds. Let $e_m^{(c)}$ be the $m$-th standard basis vector of $\RR^c$ and assume without loss of generality that $I = [\delta]$ and $H=[\ell]$. For ease of notation we set $\gamma:=(k+r-n)_+$. Choose distinct $0<\lambda_i<1$ and $0<\mu_h<1$ for $i \in I$ and $h \in H$ and define 
  \begin{align*}
      &g_i = e_i^{(k)} \ \quad \quad \quad \ \,\textrm{for } i=1,\dots,\gamma\, , \\
      &z_i = e_{\gamma + i}^{(k)} \quad \quad \quad \ \,\textrm{for } i=1,\dots,\delta \, , \\
      &w_h = e_{j+h}^{(k)}  \quad \quad \ \ \ \textrm{for } h=1,\dots,\ell \,.
  \end{align*}
  Thus, the columns of $Z$ are images of these standard basis vectors of $\RR^k$. We define them as
  \[\begin{cases}
      Zg_i = (0, e_i^{(r)},0) & \text{for } 1\leq i\leq \gamma \\
      Zz_i = (\sqrt{1-\lambda_i}e_i^{(n-r-s)}, \sqrt{\lambda_i}e_{\gamma+i}^{(r)},0) & \text{for } 1\leq i\leq (k-s)_+-\gamma \\
      Zz_i = (0,\sqrt{\lambda_i}e_{\gamma+i}^{(r)}, \sqrt{1-\lambda_i}e_{\gamma-(k-s)_+ + i}^{(s)}) & \text{for } (k-s)_+-\gamma< i\leq j-\gamma \\
      Zw_h = (\sqrt{1-\mu_h}e_{(k-s)_+-\gamma+h}^{(n-r-s)},0,\sqrt{\mu_h}e_{j-(k-s)_++h}^{(s)}) & \text{for }1\leq h\leq (k-r)_+ \\
      Zw_h = (0, \sqrt{1-\mu_h}e_{j+h-(k-r)_+}^{(r)}, \sqrt{\mu_h}e_{j-(k-s)_++h}) & \text{for }(k-r)_+<h\leq \ell
  \end{cases}\,.
  \] 
One can check that the ranges of indices are well defined. Moreover, with these definitions $Z^\top Z = {\rm Id}_k$ and $Q=ZZ^\top$ is an orthogonal projection matrix of rank $k$. By construction, the $g_i$'s are eigenvectors with eigenvalue $1$ for $Z_2^\top Z_2$ and the $z_i$'s (respectively $w_h$'s) are eigenvectors with eigenvalue $\lambda_i$ (respectively $\mu_h$)  for $Z_2^\top Z_2$ (respectively $Z_3^\top Z_3$). They are linearly independent by construction, proving that \eqref{eq:linind2} holds for this $Q$.  
\end{proof}

Suppose now that $n = r + s$. 
 By Proposition~\ref{prop:SVD},  the eigenvectors $u_i$ and $v_h$ come in singular vector pairs; index them so that $(u_i, v_i)$ is a singular vector pair for $i \in [\min\{k,s\}]$.
\begin{lemma} \label{lem:n=r+s}
Suppose $n = r + s$.
If $P_{22}$ projects onto $F \oplus  \bigoplus_{u \in I}U_i$
and $P_{33}$ projects onto $\bigoplus_{h \in H}V_h$, then $(P_{22}, P_{33})$ is GD critical for $Q$ if and only if $I \cap H = \emptyset$.
\end{lemma}
\begin{proof}
With the notation above, now $Z = \begin{bmatrix} Z_2 \\ Z_3 \end{bmatrix}$. Both $Z_2^\top Z_2$ and $Z_3^\top Z_3$ are positive semidefinite matrices. Moreover, since $r+s = n$, $s\leq r$, and $k \leq n/2$, we have $k \leq r$. Hence,
for generic $Q$ the $k \times k$ matrix $Z_2^\top Z_2$ has rank $k$, i.e., it has only positive eigenvalues. The rank of $Z_3^\top Z_3$ is $\min\{k,s\}$. Let $u \in \RR^k$ such that $Z_2^\top Z_2 u = \lambda u$. Then $Z_3^\top Z_3 u = u -Z_2^\top Z_2u = (1-\lambda)u$. We already saw that the eigenvalues of $Z_2^\top Z_2$ are at most $1$, and the eigenvalues of $Z_3^\top Z_3$ are less than $1$. It also follows that the eigenspace of $Z_2^\top Z_2$ with eigenvalue $1$ equals the kernel of $Z_3^\top Z_3$ and therefore has dimension $(k-s)_+$. Finally if  $\lambda_i \in (0,1)$ is an eigenvalue of $Z_2^\top Z_2$, then $1-\lambda_i$ is an eigenvalue of $Z_3^\top Z_3$ with the same eigenvector. To each such eigenvector $z_i \in \RR^k$ there correspond eigenvectors $u_i \in \RR^r$ of $Q_{22}$ and $v_i \in \RR^s$ of $Q_{33}$ given by
  \begin{align*}
      u_i = \frac{Z_2z_i}{\sqrt{\lambda_i}}, \quad v_i = \frac{Z_3 z_i}{\sqrt{1-\lambda_i}}\,.
  \end{align*}
  A quick computation shows that $Q\hat{u}_i = \sqrt{\lambda_i}Zz_i$ and $Q\hat{v}_i = \sqrt{1-\lambda_i}Zz_i$. Hence these are not linearly independent and the corresponding projection matrix $P$ lies in the cut locus of $Q$. Therefore such pairs must be avoided in the construction of $P$. If, instead, exactly one of $u_i$, $v_i$ is chosen for every $i$, we obtain an eigenbasis of $Z_2^\top Z_2$ and the vectors are linearly independent. This proves the result. 
\end{proof}
With all these preliminary results, now we are ready to prove Theorem
\ref{thm:prod_Grassmannian}. 
\begin{proof}[Proof of Theorem \ref{thm:prod_Grassmannian}:] 
The case $s = \ell = 0$ follows from Corollary~\ref{cor:small-grass}.
Lemma \ref{lem:ED-critpoints-of-products-of-Grassmannians} immediately implies that the ED degree of $\mathrm{pGr}(j,r) \times \mathrm{pGr}(\ell,s)$ is equal to $\binom{r}{j} \cdot \binom{s}{\ell}$. 
If $n=r+s$, Lemma \ref{lem:n=r+s} implies that the GD critical points are obtained by choosing an index set $H$ of $\ell$ eigenvectors $v_h$ of $Q_{33}$ and choosing the $\delta$ eigenvectors $u_i$ indexed by $I = [\min\{k,r,n-r\}]\backslash H$.
In other words, if $n=r+s$, 
the GD degree of $\mathrm{pGr}(j,r) \times \mathrm{pGr}(\ell,s)$ is equal to $\binom{\min\{k,s\}}{\ell}$.

If $n>r+s$, Lemma \ref{lem:n>r+s} implies that the GD critical points are obtained by choosing $\delta$ eigenvectors $u_i$ of $Q_{22}$ and $\ell$ eigenvectors $v_h$ of $Q_{33}$. So the GD degree in this case is 
\begin{align*}
      \binom{\min\{k,r,n-r\}}{\delta} \cdot \binom{\min\{k,s\}}{\ell}\, .
  \end{align*}
This formula reduces to $\binom{\min\{k,r\}}{j} \cdot \binom{\min\{k,s\}}{\ell}$ if   $n \geq r+k$ and to $\binom{n-r}{\ell} \cdot \binom{s}{\ell}$ otherwise.
\end{proof}

\section{Matroid Varieties}\label{sec:positroids}
A realizable matroid $M$ of rank $k$ on $n$ elements is given by a $k\times n$ matrix $A$ of rank $k$. The matroid $M$ is characterized by the collection of $k\times k$ minors $A_I$ which vanish and those which do not vanish. The \emph{realization space} $\mathcal{M}^\circ$ of the matroid $M$ is the subset 
\[\mathcal M^\circ = \{x \in \mathrm{Gr}(k,n) \subseteq \mathbb P^{\binom n k} :  x_I = 0 \,\, \textrm{if and only if } I \mbox{ is not a basis of } M \}\]
of the Grassmannian in its Pl\"ucker embedding.
The Zariski closure $\mathcal{M}$ is defined as
\[\mathcal M = \overline{\mathcal M^\circ} = \{x \in \mathrm{Gr}(k,n) \subseteq \mathbb P^{\binom n k} :  x_I = 0 \,\, \textrm{if } I \mbox{ is not a basis of } M \}.\]
We refer to the closure $\mathcal M$ as the \emph{matroid variety} of $M$.
Mnëv's Universality Theorem implies that these realization spaces can be arbitrarily complicated. 
In this article, all matroid varieties we consider are irreducible.

Our first result shows that several properties are invariant under matroid isomorphism. 
If $M_1, M_2$ are matroids with ground set $[n]$, then we say $M_1$ and $M_2$ are isomorphic if there is a permutation $\sigma$ of $[n]$ that maps the bases of $M_1$ to the bases of $M_2$.
\begin{lemma}
    If two matroids $M_1$ and $M_2$ are isomorphic, then their matroid varieties $\mathcal{M}_1$ and $\mathcal{M}_2$ have the same degree, ED degree, and GD degree.
\end{lemma}
\begin{proof}
    A matroid isomorphism induces a permutation on the Pl\"ucker coordinates. The first result follows from the fact that degree is invariant under linear change of coordinates. 
    For the ED and GD degree, we consider the image of the matroid variety in projection coordinates. 
    If $A_1$ is a matrix realizing $M_1$, we permute the columns by $\sigma$ to obtain a matrix $A_2$ realizing $M_2$.
    The same statement holds for projection matrices with simultaneous permutation of rows and columns. 
    Since permutation matrices are orthogonal, and the ED and GD optimization problems are $\mathrm{O}(n)$-equivariant, the result follows. 
\end{proof}
In the next subsections, we present results on the ED and GD degrees of matroid varieties of rank $2$ on $n$ elements and matroid varieties of rank $3$ on $6$ elements.
All of these matroids are realizable, and their matroid varieties are irreducible \cite[Proposition 4.2]{CL}.

\subsection{Rank Two Matroids}
A matroid $M$ of rank $2$ is determined by a set $L \subseteq [n]$ of \emph{loops} and a set partition $S_1, S_2, \ldots, S_r$  of $[n] \backslash L$ where $r\geq 2$; the sets $S_1, \ldots, S_r$ are called \emph{parallel classes}.
The bases of $M$ are two-element sets
 $\{a,b\}$  with $a \in S_i$ and $b \in S_j$ for $i \neq j$; see \cite[Section 1.5]{Oxley}.
 
 We consider these matroids up to isomorphism.
 Let $S_1, \ldots, S_r, L$ be the data of a rank $2$ matroid on $n$ elements and set $\pi_i = |S_i|$ for $i = 1, \ldots, r$.
 Since we are working up to isomorphism, we may assume that $\pi_1 \geq \cdots \geq \pi_r$. 
 We use as a representative the matroid with parallel classes $S_1 = \{1, \ldots, \pi_1\}$ and $S_i = \{\sum_{j=1}^{i-1}\pi_{j} + 1, \ldots, \sum_{j=1}^{i}\pi_{j}\}$ for $i = 2, \ldots, r$ and loops 
 $L = \{\sum_{j=1}^{r}\pi_{j} + 1, \ldots, n\}$. 
  We therefore identify an equivalence class of matroids with the integer $\ell = |L|$ and the integer partition $\pi \vdash n - \ell$.
 We denote this matroid representative by $M_{\pi, \ell}$ and we write $M_\pi = M_{\pi, 0}$.

 The matroid $M_{\pi, \ell}$ is realized by the $2 \times n$ matrix $A = A'\cdot U$ where 
 \begin{align}\label{eq:matroid-param}
 A' = 
     \begin{bmatrix}
         t_1 & \cdots & t_r \\
         s_1 & \cdots & s_r\\
     \end{bmatrix},
    && U = 
     \left [\begin{array}{ccccccccccccccccc}
    u_1^1 & u_2^1 & \cdots & u_{\pi_1}^1 & & & & & & 0\\
     & & & & \ddots & & & & & &\ddots\\
     & & & & & u_1^r & u_{2}^r & \cdots & u_{\pi_r}^r& & & 0\\
\end{array} \right ].
 \end{align}
  We write $\mathcal M_{\pi, \ell}$ for the corresponding matroid variety. Before turning to the ED and GD optimization problems, we first describe the dimension, degree, and singular locus of $\mathcal{M}_{\pi, \ell}$. 
 \begin{proposition} \label{prop:matroid-degree}
    The variety $\mathcal M_{\pi, \ell} \subseteq \mathrm{Gr}(2, n)$ has  dimension $n - \ell + r - 4$.
    Its degree is 
    \begin{align*}
        \sum_{j_1 = 0}^{\lfloor (\pi_1 - 1)/2\rfloor}
        \cdots 
        \sum_{j_r = 0}^{\lfloor (\pi_r - 1)/2\rfloor}
        (-1)^j
        C_{n - \ell - j - 2} 
        \prod_{i=1}^r
        \left (\!\! \binom{\pi_i}{j_i}\!\!\right )
    \end{align*}
    where $j = j_1 + \cdots + j_r$ and $C_{m} = \frac{1}{m+1}\binom{2m}{m}$ is the $m$-th Catalan number. 
\end{proposition}
\begin{proof}
    The parametrization \eqref{eq:matroid-param} is injective in an affine open of $\mathcal M_{\pi, \ell}$ after fixing $t_1 = s_r = 1$, $s_1 = t_r = 0$, and $u_1^i = 1$ for $i \in [r]$.
    The parameter count $n - \ell + r-4$ gives the dimension. 

    Since lifting a projective variety to a higher dimensional ambient space does not alter the degree, it suffices to consider the loopless case and replace $n$ with $n - \ell$ in the final formula. 

    We work in the Chow ring of $\mathrm{Gr}(2,n)$ \cite[Chapter 4.3]{3264}.
    The class of $\mathcal M_\pi$ is
    \begin{align*}
        [\mathcal M_\pi] = 
        h_{\pi_1 - 1}(x,y) \cdots h_{\pi_r-1}(x,y)
    \end{align*}
    where $h_i(x,y) = x^i + x^{i-1}y + \cdots + y^i$.
    This follows from the observation that $\mathcal M_\pi$ is the intersection of certain Schubert varieties; see the proof of \cite[Proposition 4.4]{PS}.
    We determine the degree by writing $[\mathcal M_\pi]$ in terms of the classes of smaller Grassmannians intersected with linear spaces.
    The functions $h_i$ satisfy the recursive relationship $h_i(x,y) = h_{1}(x,y)h_{i-1}(x,y) - e_{2}(x,y)h_{i-2}(x,y)$
    where $e_2(x,y) = xy$.
    Unraveling this recursion yields 
    \[h_{\pi_i-1} = \sum_{j_i = 1}^{\lfloor (\pi_i-1)/2\rfloor} 
    (-1)^{j_i}
    \left (\!\!\binom{\pi_i}{j_i} \!\!\right )
    h_1^{\pi_i-2j_i-1}
    e_2^{j_i}.    
    \]
    Writing $j = j_1 + \cdots + j_r$, the class $[\mathcal M_\pi]$ expands as 
    \begin{align*}
        [\mathcal M_\pi] = 
         \sum_{j_1 = 0}^{\lfloor (\pi_1 - 1)/2\rfloor}
        \cdots 
        \sum_{j_r = 0}^{\lfloor (\pi_r - 1)/2\rfloor}
        (-1)^j
        e_2^jh_1^{n - r - 2j}
        \prod_{i=1}^r
        \left (\!\!\binom{\pi_i}{j} \!\!\right ).
    \end{align*}
    The class of the Grassmannian $\mathrm{Gr}(2,n-j)$ is $[\mathrm{Gr}(2,n-j)] =  \sigma_{(j,j)}  = e_2^j$.  Here $h_1(x,y) = x+y$ is the hyperplane section class.
    Since intersecting with $n - r - 2j$ generic hyperplanes preserves the degree of $\mathrm{Gr}(2,n-j)$, the degree formula follows from $\deg(\mathrm{Gr}(2,n-j)) = C_{n - j - 2}$.
\end{proof}

The singular loci of matroid varieties are complicated. 
For instance, while much work has been done to classify smooth matroid varieties (e.g. \cite{CL}), the classification of the singular loci of matroid varieties is largely still an open problem.
Here, we give an explicit description in the case of rank $2$ matroid varieties.
Let $\hat \pi_i = (\pi_1, \ldots ,\pi_{i-1}, \pi_{i+1}, \ldots, \pi_r)$.
\begin{proposition} \label{prop:singular-rank-2} 
Let $\pi = (\pi_1, \ldots, \pi_r) \vdash n - \ell$.
If $r = 2$, then $\mathcal M_{\pi, \ell}$ is smooth. 
If $r \geq 3$, the singular locus of $\mathcal M_{\pi, \ell}$ is isomorphic (via matroid isomorphism) to
\begin{align*}
    \bigcup _{i \in [r], \pi_i \geq 2} \mathcal M_{\hat \pi_i, \ell + \pi_i}.
\end{align*}
The singular locus has an analogous description in projection coordinates.
\end{proposition}
\begin{proof}
It suffices to prove the case $\ell = 0$.
Let $V_{\pi_i} \subseteq \mathbb C^{2 \times \pi_i}$ denote the affine variety of matrices with rank $\leq 1$. 
We define
\[V_\pi = \prod_{i = 1}^r V_{\pi_i} = 
\{[A_1 \cdots A_r] \in \mathbb C^{2 \times n} : A_i \in V_{\pi_i} \textrm{ for } i = 1, \ldots, r\}
\subseteq \mathbb C^{2 \times n}\] 
and let $V_\pi^\circ = \{A \in V_\pi: \mathrm{rank}(A) = 2\}$ be the open subset of full rank matrices in $V_\pi$.
Then the matroid variety $\mathcal{M}_{\pi}$ is the closure of the image of $V_\pi^\circ$ under the Pl\"ucker embedding.
Because the Pl\"ucker embedding is a smooth morphism on $\CC^{2 \times n}$ of relative dimension $4$ and $V_\pi^\circ$ is invariant under the action of $\mathrm{GL}(2)$, the embedding is still a smooth morphism when restricted to $V_\pi^\circ$. 
Thus the image of the singular locus of $V_\pi^\circ$ is precisely the  singular locus of the image. 
It therefore suffices to characterize the singular locus of $V_\pi^\circ$.
For the last statement, we remark that the map $A \mapsto A^\top (AA^\top)^{-1} A$ is also a smooth morphism on $V_\pi^\circ$ of relative dimension $4$, and so the same argument holds for the matroid variety in $\mathrm{pGr}(k,n)$.

If $\pi_i = 1$, then $V_{\pi_i} = \CC^2$ is smooth. 
If $\pi_i \geq 2$, then $V_{\pi_i}$ has precisely one singular point, namely the zero matrix. 
Thus a point $A = [A_1\,\, \cdots\,\,  A_r]$ in $V_\pi^\circ$ is singular if and only if $A_i = 0$ for some $i \in [r]$ with $\pi_i \geq 2$. 
This matches the claim.
If $r = 2$, then $A_1 = 0$ implies that $[A_1 \,\, A_2]$ has rank $1$ and is therefore not in $V_\pi^\circ$, so $V_\pi^\circ$ is smooth in this case.
\end{proof}

As an application of this lemma, we show that the ED degree of rank 2 matroid varieties exhibit a monotonicity property. 
\begin{theorem} \label{thm:monotonicity} Let $M_{\pi, \ell}$ and $M_{\pi', \ell'}$
be rank $2$ matroids on $[n]$ and $[n']$, respectively, such that 
$\pi_i \geq \pi_i'$ for all $\pi_i' \in \pi'$.
Then the ED degree of $\mathcal{M}_{\pi, \ell}$ is at least the ED degree of $\mathcal{M}_{\pi', \ell'}$.
\end{theorem}
\begin{proof}
It suffices to prove the statement for the case when $\ell = \ell' = 0$ and   $\pi_i = \pi_i' + 1$ for exactly one $i$ and $\pi_j= \pi_j'$ otherwise where $\pi_i'$ may be zero. 
We use the parametrization \eqref{eq:matroid-param} on the open chart $u_{1}^i = 1$.
If $\pi_i'\neq 0$, then the parametrization matrix  $A_{\pi}$ is obtained from $A_{\pi'}$ 
by adding the column $\begin{bmatrix} u^i_{\pi_i'+1} t_i & u^i_{\pi_i'+1} s_i \end{bmatrix}^\top$ to the end of block $i$. 
If $\pi_i' = 0$ and $\pi$ has one more part than $\pi'$, then we append a new column $\begin{bmatrix}  t_r & s_r \end{bmatrix}^\top$ at the end of $A_{\pi'}$ to obtain $A_{\pi}$. If the new column is inserted between column $c$ and $c+1$, we 
call this position $\hat{c}$. Then 
the cocircuit matrix
$X_{\pi}$ is obtained from $X_{\pi'}$
by inserting a new column and row at position $\hat{c}$.

Let $Q$ be a generic $n \times n$ symmetric matrix and let $\hat{Q}$ be the symmetric matrix obtained from $Q$
by inserting a column and row of zeros at position $\hat{c}$. 
Let $C_{\pi'}$ be a $2 \times n$ matrix which is critical for \eqref{eq:trace-opt} for the model $\mathcal{M}_{\pi'}$ and data $Q$. 
Let $f_{\pi'}  =\mathrm{trace}(X_{\pi'}^2)$
and $g_{\pi'} = \mathrm{trace}(X_{\pi '}^2 Q)$.
Since $C_{\pi'}$ is critical, the Jacobian matrix \eqref{eq:crit-matrix} 
\begin{align*}
    \begin{bmatrix}
        g_{\pi'}
        & \nabla_t g_{\pi'} & \nabla_s g_{\pi'}
        & \nabla_{u} g_{\pi'}\\
        f_{\pi'}
        & \nabla_t f_{\pi'} & \nabla_s f_{\pi'}
        & \nabla_{u} f_{\pi'}
    \end{bmatrix}
\end{align*}
has rank at most one when evaluated at $C_{\pi'}$. 
Here, $t = (t_1, \ldots, t_{r'})$, $s = (s_1, \ldots, s_{r'})$, and $u = (u_{2}^1, \ldots, u_{\pi_{r'}}^{r'})$.
We first claim that the $2 \times (n+1)$ matrix  $\hat{C}_{\pi}$ obtained from $C_{\pi'}$ by setting the new parameters 
in column $\hat{c}$, namely $u^i_{\pi_i'+1}$ or $t_r$ and $s_r$, equal to zero is a critical point for $\hat Q$ on $\mathcal{M}_{\pi}$.
For this, we show that the corresponding Jacobian matrix evaluated at $\hat{C}_{\pi}$ also has rank at most one.
We compute $f_\pi  = \mathrm{trace}(X_\pi^2) = \mathrm{trace}(X_{\pi'}^2) + h = f_{\pi'} + h$ where  
\begin{equation}\label{eq:h}
    h = (u^i_{\pi_i'+1})^2 h' \textrm{ for some } h' \in \mathbb R[t,s,u] \textrm{ if }\pi_i'\neq 0
    \quad \textrm{ or } \quad  
    h \in (\RR[t,s,u])[s_r, t_r]_2 \textrm{ if }\pi_i'= 0.
\end{equation}
This implies that, evaluated at $\hat{C}_\pi$,
\[ [f_\pi \,\, \nabla_t f_\pi \,\, 
\nabla_s f_\pi \, \, \nabla_u f_\pi]  = [f_{\pi'} \,\, \nabla_t f_{\pi'} \,\, 
\nabla_s f_{\pi'} \, \, \nabla_u f_{\pi'}].\]
Moreover, $\frac{\partial f_\pi}{\partial u^i_{\pi_i'+1}} =0$
in the first case, and $\frac{\partial f_\pi}{\partial t_r} = \frac{\partial f_\pi}{\partial s_r} =0$, in the second case. 
 Because of the way $\hat{Q}$ is constructed, the polynomial $g_\pi = \mathrm{trace}( X_\pi^2 \hat{Q})$ depends only on the entries of $X_\pi^2$ not in row or column $\hat{c}$.
 Indeed, a computation reveals that $(X_\pi^2)_{jl} = (X_{\pi'}^2)_{jl} + h_{jl}$
 where $h_{jl}$ is as in \eqref{eq:h}, 
 and hence $g_\pi = g_{\pi'} + \bar{h}$ where $\bar{h}$ is as in \eqref{eq:h}.  
 Hence, when evaluated at $\hat{C}_\pi$
$$ [g_\pi \,\, \nabla_t g_\pi \,\, 
\nabla_s g_\pi \, \, \nabla_u g_\pi]  = [g_{\pi'} \,\, \nabla_t g_{\pi'} \,\, 
\nabla_s g_{\pi'} \, \, \nabla_u g_{\pi'}].$$
Similarly, 
$\frac{\partial g_\pi}{\partial u^i_{\pi_i'+1}} =0$
in the first case, and $\frac{\partial g_\pi}{\partial t_r} = \frac{\partial g_\pi}{\partial s_r} =0$, in the second case. 
  This shows that the matrix in \eqref{eq:crit-matrix} with $\hat{Q}$ evaluated at 
$\hat{C}_{\pi}$ has rank at most one, and therefore 
$\hat{C}_\pi$ satisfies the conditions of being a critical point as long as it is not a singular point of $\mathcal M_\pi$.
Indeed, Proposition \ref{prop:singular-rank-2} shows that $\hat{C}_\pi$ is singular if and only if $t_i = s_i = 0$ for some $i$ with $\pi_i \geq 2$. 
Since $C_{\pi'}$ is a smooth point of $\mathcal M_{\pi'}$,  $\hat{C}_\pi$ is a smooth point of $\mathcal M_{\pi}$ by construction.
\end{proof}

We now give explicit formulae for the ED and GD degrees of some matroid varieties. 
In the case $r = 2$, these are direct consequences of Theorem~\ref{thm:prod_Grassmannian}.
Recall that these matroid varieties are determined by $\pi_1, \pi_2 > 0$ and $\ell \geq 0$ with $\pi_1 + \pi_2 + \ell = n$.

\begin{corollary}\label{cor:matroidED}
 The ED degree of $\mathcal{M}_{(\pi_1, \pi_2), \ell}$ is $\pi_1\pi_2$.
 If $\ell=0$, the GD degree of $\mathcal{M}_{(\pi_1, \pi_2)}$ is $2$ if 
 $n - 2 \geq \pi_1 \geq 2$ and $1$ otherwise. If $\ell>0$, the GD degree of $\mathcal{M}_{(\pi_1, \pi_2),\ell}$ is $2^{\#\{\lvert\pi_i\rvert>1\}}$.
\end{corollary}
Table \ref{tab:placeholder}
displays the ED degrees of 
rank 2 matroid varieties in $\mathrm{pGr}(2,n)$ for $n=4,5,6$. The starred degrees are consequences of Corollary~\ref{cor:matroidED}. For any $n$, the matroid variety corresponding to the partition $1^n$ is equal to the Grassmannian $\mathrm{pGr}(2,n)$ itself. The ED degree is $\binom{n}{2}$ in this case. The rest of the entries of the table are numerically computed. 
\begin{table}[H]
    \centering
    \begin{gather*}
    \begin{array}{c|cccc|cccccc}
    \toprule
    \textrm{Partition } \pi & 1^4 & 21^2 & 2^2 & 31 &1^5 & 21^3 & 2^21 & 31^2 & 32 & 41 \\
    \midrule
    \textrm{ED Degree of } \mathcal M_\pi    & 6  & 10 & 4^* & 3^* & 10  & 24 & 44 & 16 & 6^* & 4^*\\
    \bottomrule
    \end{array}\\
    \begin{array}{c|ccccccccccccc}
    \toprule
    \textrm{Partition } \pi & 1^6 & 21^5 & 2^21^2 & 2^3 & 31^3 & 321 & 3^2 & 41^2 & 42 & 51 \\
    \midrule
    \textrm{ED Degree of } \mathcal M_\pi     & 15 & 44 & 120 & 168 & 39 & 76 & 9^* & 22 & 8^* & 5^*\\
    \bottomrule
    \end{array}
    \end{gather*}
    
    \caption{Numerically computed ED degrees of matroid varieties in $\mathrm{pGr}(2,n)$ for $n = 4,5$ (above) and $n= 6$ (below). 
    The starred degrees are consequences of Corollary~\ref{cor:matroidED}.}
    \label{tab:placeholder}
\end{table}

We now give a general formula for the GD degree when $\pi$ has at least $3$ parts. 
We use the parametrization $A = A' \cdot U$ where $A',U$ are as in \eqref{eq:matroid-param}, and
we work in the affine chart $u^i(u^i)^\top = 1$ for all $i$. 
Now the projection matrices $P$ and $P'$ for $A$ and $A'$ satisfy the relation $P = U^\top P'U
$. 
Using the cyclic property of the trace, we replace the GD optimization problem with the following constrained optimization problem with data $Q \in \mathrm{pGr}(2,n)$
\begin{align}\label{eq:factored-matroid-opt}
    \textrm{maximize}\quad \mathrm{trace}(UQU^\top P') 
    \quad \quad 
    \textrm{subject to} \quad  UU^\top  = 
    \mathrm{Id}_r, \, P' \in \mathrm{pGr}(2,r).
\end{align} 
We will prove our claim by optimizing separately with respect to $P'$ and $U$.
We use the following lemma to show that this is possible.
\begin{lemma}\label{lem:envelope}
Let $X \times Y\subseteq \mathbb C^c \times \mathbb C^d$ be an affine variety such that $X,Y$ have smooth real points. 
Suppose that $f: X \times Y \to \mathbb C$ with $f(X_\RR \times Y_\RR) \subseteq \RR$ is a differentiable function such that for all $\hat x \in X$, the problem $ \max_{y \in Y_{\mathbb R}} f(\hat x, y)$ has a unique critical point $y^*(\hat x) \in Y$. 
Then $g(x) = f(x, y^*(x))$ is differentiable, and a pair $(\hat x, \hat y) \in X \times Y$ is critical for 
\begin{align*}
    \max_{(x,y) \in X_\mathbb R \times Y_\mathbb R} f(x,y)
\end{align*}
if and only if $\hat y = y^*(\hat x)$ and $\hat x$ is critical for the optimization problem  $\max_{x \in X_\RR} g(x)$.
\end{lemma}
\begin{proof}
    We first observe that $y^*: X \to Y$ is differentiable by the implicit function theorem, so the first claim follows. 
    The pair $(\hat x, \hat y)$ is critical if and only if 
    \begin{align*}
        \nabla f(\hat x, \hat y) \in N_{\hat x}X \times N_{\hat y} Y,
    \end{align*}
     where $N_{\hat x}X$ and $N_{\hat y}Y$ refer to the normal spaces of $X$ and $Y$ at the points $\hat x$ and $\hat y$, respectively.  
    This means that $\nabla_xf(\hat x, \hat y) \in N_{\hat x}X$ and
    $\nabla_yf(\hat x, \hat y) \in N_{\hat y}Y$.
    By assumption, the condition $\nabla_yf(\hat x,\hat y) \in N_{\hat y}Y$ implies that $\hat y = y^*(\hat x)$.
    Therefore $(\hat x, \hat y)$ is critical if and only if $\hat y = y^*(\hat x)$ and $\nabla_xf(\hat x, y^*(\hat x)) \in N_{\hat x}X$.

    On the other hand, $\hat x$ is critical for $\max_{x \in X_\RR} g(x)$ if and only 
    \begin{align*}
    \nabla g(\hat x) = 
        \nabla_x f(\hat x, y^*(\hat x)) + \mathrm{Jac}_{y^*}(\hat x) \cdot \nabla_y
        f(\hat x, y^*(\hat x)) \in 
        N_{\hat x}X.
    \end{align*}
    By the definition of $y^*$, $\nabla_y f(\hat x, y^*(\hat x)) \in N_{y^*(\hat x)}Y$ and hence is in the kernel of $\mathrm{Jac}_{y^*}(\hat x)$. 
    Thus $\hat x$ is critical for $\max_{x \in X_\RR} g(x)$ if and only if $\nabla_xf(\hat x, y^*(\hat x))$ is in $N_{\hat x}X$. 
    Together with the previous paragraph, this completes the proof.
\end{proof}
In order to apply Lemma~\ref{lem:envelope}, we prove that if we fix $U$ in \eqref{eq:factored-matroid-opt}, the optimization has a unique critical point. 
This is the content of the next result.
\begin{lemma}\label{lem:unique-cp}
For generic data $Q \in \mathrm{pGr}(2,n)$ and fixed $U$ as in \eqref{eq:matroid-param}, the problem  \begin{align*}
    \mathrm{maximize}\quad \mathrm{trace}(UQU^\top P') \quad\quad \mathrm{subject\,\, to} \quad 
     P' \in \mathrm{pGr}(2,r) 
\end{align*}
has a unique critical point such that $U^\top P'U$ is not in the cut locus of $Q$.
\end{lemma}
\begin{proof}
We first assume $\ell = 0$.
By Proposition \ref{prop:cps}, each critical point $P'$ projects onto an invariant subspace of $UQU^\top $. 
Since $UQU^\top $ has rank $2$, there is one critical point $P'$ with the property $\mathrm{rank}(UQU^\top P')=2$.
It therefore suffices to prove that $\mathrm{rank}(UQU^\top P')=2$ is equivalent to $\mathrm{rank}(QU^\top P'U)=2$.  
Since $U$ has full row rank $r$, ${\rm im}(P'U)={\rm im}(P')$. Hence ${\rm rank}(QU^\top P')={\rm rank}(QU^\top P'U)$. If ${\rm rank}(UQU^\top P')=2$, in particular ${\rm rank}(QU^\top P')=2$ and ${\rm rank}(QU^\top P'U)=2$. Conversely, let ${\rm rank}(QU^\top P'U)=2$. Then $QU^\top $ is injective on ${\rm im}(P')$ and ${\rm im}(QU^\top P')={\rm im}(Q)$. Since ${\rm rank}(QU^\top )={\rm rank}(UQ)$, $U$ is injective on ${\rm im}(Q)$. 
Hence ${\rm rank}(UQU^\top P')=2$.
Note that we only used $\mathrm{rank}(Q) = 2$ and not $Q^2 = Q$.
Since rank is preserved by taking sufficiently large submatrices, the proof extends to the case $\ell > 0$.
\end{proof}

\begin{theorem}
    Suppose $r \geq 3$.
    For real data $Q$, the critical points of the optimization problem \eqref{eq:factored-matroid-opt}  are pairs $(U, P')$ where each row $u^i$ of $U$ is an eigenvector with nonzero eigenvalue of the block $Q_{ii}$, and $P'$ is the unique critical point in Lemma \ref{lem:unique-cp}. 
\end{theorem}

\begin{proof}
By Lemmas~\ref{lem:envelope} and \ref{lem:unique-cp}, the pair $(U, P')$ is critical for \eqref{eq:factored-matroid-opt} if $U$ is critical for
\begin{align}\label{eq:env-opt}
    \textrm{maximize  }\mathrm{trace}(UQU^\top P^*(U))
     \textrm{   subject to  }UU^\top  = \mathrm{Id}_r 
\end{align}
and $P'$ is as in Lemma~\ref{lem:unique-cp}.
It is a linear algebra exercise to show that \[\mathrm{trace}(UQU^\top P^*(U))  = \mathrm{trace}(UQU^\top )  = 
\sum_{i=1}^r u^i Q_{ii} (u^i)^\top .
\] 
Thus the critical points \eqref{eq:env-opt} are matrices $U$ where $u^i$ is an eigenvector of the block $Q_{ii}$. 
The proof is concluded by showing that for such a $U$, the point $P = U^\top P'U$ is a singular point in the model if and only if $u^iQ_{ii}= 0$ for some $i \in[r]$.
For sufficiently general $Q$, $u^iQ_{ii}= 0$ occurs only when $\pi_i > 2$.
Indeed, since $u^i$ is an eigenvector of $Q_{ii}$, $u^iQ_{ii} = 0$ is equivalent to $u^iQ_{ii}(u^i)^\top = 0$. 
It follows that $u^i Q_{ii} = 0$ is equivalent to the $i$th row of $UQU^\top $ being equal to zero. 
Since the rows of $A'$ span the image of $UQU^\top$, this condition is equivalent to the $i$th column of $A$ being zero, which is precisely the description of the singular locus in Proposition \ref{prop:singular-rank-2}.
Finally, we observe that these critical points are not in the cut locus of $Q$ by Lemma~\ref{lem:unique-cp}.
\end{proof}
\begin{corollary}\label{cor:GD-matroid}
    If $r \geq 3$, then the GD degree of $\mathcal{M}_{\pi, \ell}$ is $2^{\#\{i \in [r] : |\pi_i| > 1\}}$.
\end{corollary}
\begin{proof}
    The problem \eqref{eq:factored-matroid-opt} has $2^r \cdot 2^{\#\{i \in [r] : |\pi_i| > 1\}}$ critical points, since every block $Q_{ii}$ of $Q$ has two eigenvectors with nonzero eigenvalue if $\pi_i \geq 2$. 
    The factor $2^r$ comes from the fact that every chosen eigenvector may have its sign flipped. 
    However, the map $(A', U) \mapsto A = A'U$ is $2^r$-to-one, since we can simultaneously flip the sign of a row of $U$ and a column of $A'$.
\end{proof}

\begin{table}[!ht]
    \centering
    \scalebox{0.95}{
    $
    \begin{array}{cc|cccc||cc|cccc}
    \toprule
       & & \textrm{\small{Dim}} & \textrm{\small{Deg}} & \textrm{ED} & \textrm{GD}
       & & & \textrm{Dim} & \textrm{Deg} & \textrm{ED} & \textrm{GD}\\
    \midrule 
    1. &
    
        \mathcal S_{456} & 0 & 1 & 1_{\mathbb R} & 1_{\mathbb R} & 20. &
      
         \mathcal S_{136} & 5 & 5 &24 & 3_{\mathbb R}
        \\[12pt]   2. &
      \mathcal S_{356} & 1 & 1 & 2_{\mathbb R} & 2_{\mathbb R} & 21. &
       \mathcal S_{235}& 5 & 5 &  24 & 3_{\mathbb R}
        
        \\[12pt] 3. &
        \mathcal S_{256} & 2 & 1 & 3_{\mathbb R} & 3_{\mathbb R} & 22. &
       
          \mathcal S_{145} & 5 & 6 & 32 & 3_{\mathbb R}
        \\[12pt] 4. &
        \mathcal S_{346} &  2& 1 & 3_{\mathbb R} & 3_{\mathbb R} & 23. &
         
        \scalebox{0.6}{
        $\begin{bmatrix}
            1 & a & 0 & 0 & 0 & c\\
            0 & 0 & 1 & b & 0 & d\\
            0 & 0 & 0 & 0 & 1 & e
        \end{bmatrix}$} & 5 & 11 &80 & 4_{\mathbb R}
        \\[12pt] 5. &
        \scalebox{0.6}{
        $\begin{bmatrix}
            0 & 1 & a & 0 & 0 & 0\\
            0 & 0 & 0 & 1 & b & 0\\
            0 & 0 & 0 & 0 & 0 & 1
        \end{bmatrix}^{*}$} & 2 & 2 & 4_{\mathbb R} & 4_{\mathbb R} & 24. &
          \scalebox{0.6}{
        $\begin{bmatrix}
            1 & a & 0 & b & 0 & c\\
            0 & 0 & 1 & d & 0 & 0\\
            0 & 0 & 0 & 0 & 1 & e
        \end{bmatrix}$} & 5 & 11 &80 & 4_{\mathbb R}
        \\[12pt] 6. &
        \mathcal S_{345} & 3 & 1 & 4_{\mathbb R} & 1_{\mathbb R} & 25. &
        
         \scalebox{0.6}{
        $\begin{bmatrix}
            1 & 0 & a & 0 & 0 & ac\\
            0 & 1 & b & 0 & 0 & bc\\
            0 & 0 & 0 & 1 & d & e
        \end{bmatrix}$} & 5 & 16 & 200 & 92
        \\[12pt] 7. &
        \mathcal S_{156} & 3 & 1 & 4_{\mathbb R} & 1_{\mathbb R} &
        26. &
        \scalebox{0.6}{
        $\begin{bmatrix}
            1 & 0 & 0 & a  & b & 0\\
            0 & 1 & 0 & ac  & 0 & cd\\
            0 & 0 & 1 & 0  & be & -de
        \end{bmatrix}$} & 5 & 42 & 1348 &  1016  
        \\[12pt] 8. & 
        \scalebox{0.6}{
        $\begin{bmatrix}
            0 & 1 & a & 0 & 0 & 0\\
            0 & 0 & 0 & 1 & 0 & b\\
            0 & 0 & 0 & 0 & 1 & c
        \end{bmatrix}^{*}$} & 3 & 3 & 6_{\mathbb R} & 6_{\mathbb R}&
         27. &
         \mathcal S_{234} & 6 & 5 & 10_{\mathbb R} & 1_{\mathbb R}
        \\[12pt] 9. & 
        \scalebox{0.6}{
        $\begin{bmatrix}
            1 & a & 0 & 0 & 0 & 0\\
            0 & 0 & 1 & b & c & 0\\
            0 & 0 & 0 & 0 & 0 & 1
        \end{bmatrix}^{*}$} & 3 & 3 & 6_{\mathbb R} & 6_{\mathbb R} & 28. &
         \mathcal S_{126} & 6 & 5 & 10_{\mathbb R} & 1_{\mathbb R}
        \\[12pt] 10. & 
        \scalebox{0.6}{
        $\begin{bmatrix}
            1 & a & 0 & 0 & 0 & 0\\
            0 & 0 & 1 & b & 0 & 0\\
            0 & 0 & 0 & 0 & 1 & c
            \end{bmatrix}^{*}$} & 3 & 6 & 8_{\mathbb R} & 8_{\mathbb R} & 29. &
       
        \scalebox{0.6}{
        $\begin{bmatrix}
            1 & 0 & a & 0 & 0 & b\\
            0 & 1 & c & 0 & 0 & d\\
            0 & 0 & 0 & 1 & e & f
        \end{bmatrix}$} & 6 & 21 & 120 & 6_{\mathbb R}
        \\[12pt] 11. &
        \mathcal S_{246} & 3 & 2 & 10 & 10 & 30. &
        \scalebox{0.6}{
        $\begin{bmatrix}
            1 & 0 & a & b & 0 & c\\
            0 & 1 & d & e & 0 & 0\\
            0 & 0 & 0 & 0 & 1 & f
        \end{bmatrix}$} & 6 &21 & 120 & 6_{\mathbb R}
        \\[12pt] 12. &
        \mathcal S_{236} & 4 & 2 & 6_{\mathbb R} & 2_{\mathbb R} & 31. &
        \mathcal S_{135} & 6 & 16 & 104 & 10
        \\[12pt] 13. &
        \scalebox{0.6}{
        $\begin{bmatrix}
            1 & c & d & 0 & 0 & 0\\
            0 & 0 & 0 & 1 & 0 & a\\
            0 & 0 & 0 & 0 & 1 & b
        \end{bmatrix}^{*}$} & 4 & 6 & 9_{\mathbb R} & 3_{\mathbb R} & 32. &  
        \scalebox{0.6}{
        $\begin{bmatrix}
            1 & 0 & 0 & 0 & a & b\\
            0 & 1 & 0 & d & 0 & c\\
            0 & 0 & 1 & e & f & 0
        \end{bmatrix}$} & 6 &  42 & 1070 & 506
        \\[12pt] 14. &
        \mathcal S_{245} & 4& 3 &16 & 2_{\mathbb R} & 33. &
          \mathcal S_{134} & 7 & 21 & 52 & 2_{\mathbb R}
        \\[12pt] 15. &
         \mathcal S_{146} & 4& 3 &16 & 2_{\mathbb R} & 34. &
       \mathcal S_{125} & 7 & 21 & 52 & 2_{\mathbb R}
        \\[12pt] 16. &
        \scalebox{0.6}{
        $\begin{bmatrix}
            1 & a & 0 & 0 & 0 & 0\\
            0 & 0 & 1 & b & 0 & c\\
            0 & 0 & 0 & 0 & 1 & d
        \end{bmatrix}$} & 4 & 8 & 20 & 20 & 35. & 
        \scalebox{0.6}{
        $\begin{bmatrix}
            1 & 0 & 0 & 0 & a & b\\
            0 & 1 & 0 & c & e & f\\
            0 & 0 & 1 & d & eg & fg
        \end{bmatrix}$} & 7 &  42 & 140 & 6_{\mathbb R}
        
        \\[12pt] 17. &
        \scalebox{0.6}{
        $\begin{bmatrix}
            1 & a & 0 & 0 & c & 0\\
            0 & 0 & 1 & b & d & 0\\
            0 & 0 & 0 & 0 & 0 & 1
        \end{bmatrix}$} & 4 & 5 & 44 & 34 & 36. &
        \scalebox{0.6}{
        $\begin{bmatrix}
            1 & 0 & 0 & 0 & a & b\\
            0 & 1 & 0 & c & d & e\\
            0 & 0 & 1 & f & g & 0
        \end{bmatrix}$} & 7 & 42 & 416 & 34
        
         \\[12pt] 18. &
         \scalebox{0.6}{
        $\begin{bmatrix}
            0 & 1 & 0 & a  & 0 & c\\
            0 & 0 & 1 & b  & 0 & 0\\
            0 & 0 & 0 & 0  & 1 & d
        \end{bmatrix}$} & 4 & 5 & 44 & 34 & 37. &
        \mathcal S_{124} & 8 & 42 & 102 & 3_{\mathbb R}
        \\[12pt]
        19. &
        \scalebox{0.6}{
        $\begin{bmatrix}
            1 & a & 0 & 0 & 0 & 0\\
            0 & 0 & 1 & 0 & b & c\\
            0 & 0 & 0 & 1 & d & e
        \end{bmatrix}^{*}$} & 5 & 10 & 12_{\mathbb R} & 2_{\mathbb R} &
        38. &
         \mathcal S_{123} & 9 & 42 & 20 & 1_{\mathbb R}
        \\
        \bottomrule
    \end{array} $}
    \caption{Birational parametrizations of matroid realization spaces up to matroid isomorphism, followed by the dimensions, degrees, ED degrees, and GD degrees of these varieties.}
    \label{tab:36-matroids}
\end{table}

\subsection{Rank 3 Matroids on 6 Elements}
There are 2053 matroids of rank $3$ on $6$ elements; up to isomorphism there are $38$ \cite[Section 5]{dukes}. 
They are all realizable and their matroid varieties are  irrreducible.
Of these 38 matroid varieties, 20 are Schubert varieties, which we discuss in Section~\ref{sec:schubert}. 
We use the notation $\mathcal S_{ijk}$ introduced in the subsequent section; this identifies a Schubert variety by the three pivot columns in its reduced row echelon form. 
Of the remaining $18$ matroid varieties, 5 are products of smaller Grassmannians; these are marked with an  asterisk in Table~\ref{tab:36-matroids}. 

\begin{theorem}\label{thm:table}
    The degrees, ED degrees, and GD degrees of rank 3 matroid varities on 6 elements are found in Table~\ref{tab:36-matroids}.
\end{theorem}
\begin{proof}[Proof and discussion.]
    Of the 38 rank $3$ matroids on $6$ elements, $37$ are positroids; a positroid is a matroid that has a point in its realization space whose Pl\"ucker coordinates all have the same sign. 
    We computed these positroids using Grassmann necklaces \cite{postnikov}. 
    The matroid which is not isomorphic to a positroid (row 26) was found using Gemini 3.1. 
    We check that our 38 representative matroids are nonisomorphic using \verb+is_isomorphic+ in \verb+Oscar.jl+ \cite{oscar}.
    To obtain the parametrizations, we used the \verb+Oscar.jl+  function \verb+realization_space+.

    The dimensions and degrees were computed using \verb+Oscar+.
    The ED and GD degrees, along with the real critical point counts in rows 8, 13, and 19  follow from Theorem~\ref{thm:prod_Grassmannian}; those in rows 1, 2, 3, 4, 7, 12, and 28 will follow from Corollary~\ref{cor:Schubert-EDGD} in Section~\ref{sec:schubert}.
    The remaining ED and GD degrees and real critical point counts were computed numerically using \verb+HomotopyContinuation.jl+. 
    The GD degrees were computed by solving the ED problem, and tracking the generic data matrix to a generic data matrix in the Grassmannian via parameter homotopy. 
    From these start solutions, we use monodromy to find the remaining solutions.
    A subscript $\RR$ in the table means all points are real for real data.
    We verified the ED and GD degrees in rows 1--16 symbolically in \verb+Oscar.jl+. 
\end{proof}

\section{Schubert Varieties}\label{sec:schubert}

We conclude with a discussion of Schubert varieties.  
Schubert cells stratify the Grassmannian $\mathrm{Gr}(k,n)$ based on the reduced row echelon forms of the $k \times n$ matrices. 
More concretely, given a subset $\{i_1 <  \cdots < i_k\}$ of $[n]$, the corresponding \emph{Schubert cell} $\mathcal S_{i_1\cdots i_k}^\circ$ contains all vector spaces in $\mathrm{Gr}(k,n)$ with basis whose reduced row echelon form has pivot columns $\{i_1, \ldots, i_k\}$. 
The corresponding \emph{Schubert variety} $\mathcal S_{i_1 \cdots i_k} = \overline{\mathcal S_{i_1\cdots i_k}^\circ}$ is the Zariski closure of the Schubert cell. 
For example $\mathcal S_{135} \subseteq{\mathrm{Gr}(3,6)}$ is parameterized by matrices of the form 
\begin{align*}
    \begin{bmatrix}
        1 & a & 0 & b & 0 & c\\
        0 & 0 & 1 & d & 0 & e\\
        0 & 0 & 0 & 0 & 1 & f
    \end{bmatrix}.
\end{align*}
Schubert varieties are special cases of  matroid varieties which we studied in Section~\ref{sec:positroids}. 
The vanishing ideal in Pl\"ucker coordinates of the Schubert variety $\mathcal S_{i_1\cdots i_k}$ is 
\begin{equation}\label{eq:schub-ideal}
    I_{\mathcal S_{i_1\cdots i_k}} = 
    I_{\mathrm{Gr}(k,n)} + \langle x_{j_1 \cdots j_k} : j_1< \cdots < j_k, 
    \textrm{ there exists } \alpha \in [k] \textrm{ such that }
    j_{\alpha} < i_\alpha \rangle
\end{equation}
where $I_{\mathrm{Gr}(k,n)}$ is the ideal generated by the Pl\"ucker relations. 
\begin{remark}
    We have defined Schubert varieties with respect to the flag $\langle e_n \rangle \subseteq \langle e_{n-1}, e_n \rangle \subseteq \cdots \subseteq \langle e_1, \ldots, e_n \rangle$.
    Because of the $\mathrm{O}(n)$-equivariance of the optimization problem \eqref{eq:trace-opt}, the ED and GD degrees are independent of the reference flag. 
\end{remark}
The ED and GD degrees of Schubert varieties were studied  for $k = 2$ in \cite[Section 6]{FRS}.
Our first result of this section is the resolution of \cite[Conjecture 6.7]{FRS}. 

\begin{theorem}\label{thm:k2-schub}
    The Schubert variety $\mathcal S_{ij}$ has GD degree $1$ if $j=i+1 \bmod n$ and $2$ otherwise.
\end{theorem}
\begin{proof}
    The Schubert variety $\mathcal S_{ij}$ is isomorphic (by permutation) to the matroid variety $\mathcal M_{\pi, i-1}$ where $\pi = (j - i, 1^{n-j+1})$.
    If $j = n$, then $\pi$ has two parts and the GD degree is $1$ if $i = 1$ and $2$ otherwise by Corollary~\ref{cor:matroidED}. 
    If $j < n$, then $\pi$ has at least three parts and the GD degree is $1$ if $j = i+1$ and $2$ otherwise by Corollary~\ref{cor:GD-matroid}.
\end{proof}
 In addition to being matroid varieties, some Schubert varieties also have the structure of a product of Grassmannians in the sense of Theorem~\ref{thm:prod_Grassmannian}. 

\begin{corollary} \label{cor:Schubert-EDGD}
    The Schubert variety 
    $\mathcal S_{j+1, \ldots, j+i, j+r+1,\ldots, n}$
    with $r = n - k + i - j$
     has ED degree $\binom{r}{i}$ and GD degree
    \begin{align*}
        \begin{cases}
            \binom{\min\{r,k\}}{i} &\textrm{ if }   i \leq j \\
             \binom{k-i+j}{j}
            &\textrm{ if } i > j.
        \end{cases}
    \end{align*}
\end{corollary}
\begin{proof}
    If $r = i$, the variety is a single point, and the result is trivial. 
    Otherwise, this Schubert variety is embedded in ${\rm pGr}(k,n)$ as ${\rm pGr}(i, r) \times {\rm pGr}(k-i,k-i)$. 
    The result follows from applying Theorem~\ref{thm:prod_Grassmannian} in both cases $r \geq k - i$ and $r< k - i$.
\end{proof}

Theorem~\ref{thm:gd-schubert} states that if the pivot columns $i_1, \ldots, i_k$ are cyclically adjacent, i.e., $i_{j+1} = (i_j\mod n) + 1$ for $j \in [k-1]$, then the corresponding Schubert variety has GD degree 1, and that it is easy to compute this unqiue critical point. 
The next result makes this precise; 
it was proved for the case $k = 2$ in \cite[Theorem 6.6]{FRS}.

\begin{theorem}\label{thm:zero}
    Given a generic data point $Y \in \mathrm{Gr}(k,n)$, the minimizer of \eqref{eq:trace-opt} for the model $\mathcal M = \mathcal S_{i_1 \cdots i_k}$ with cyclically adjacent pivots is obtained by setting the Pl\"ucker coordinates $y_{j_1 \cdots j_{k}}$ satisfying the condition in \eqref{eq:schub-ideal} to zero. 
\end{theorem}
\begin{proof}
    We first prove the case $\mathcal S_{i+1, \ldots, i+k}$.
    Let $A = [A_1 \,\, A_2] \in \mathbb C^{k \times n}$ be a basis for $\mathrm{im}(Y)$ with $A_1 \in \mathbb C^{k \times i}$ and $A_2 \in \mathbb C^{k \times (n-i)}$.
    The corresponding projection matrix~is 
    \begin{align*}
        Q = \begin{bmatrix}
            A_1^\top (AA^\top)^{-1}A_1 & A_1^\top (AA^\top)^{-1}A_2\\
            A_2^\top (AA^\top)^{-1}A_1 & 
            A_2^\top (AA^\top)^{-1}A_2
        \end{bmatrix}.
    \end{align*}
    By Corollary~\ref{cor:small-grass}, the unique critical point is identified with the embedded $(n-i) \times (n-i)$ projection matrix onto $\mathrm{im}(A_2^\top (AA^\top)^{-1}A_2)$.
    Since $\mathrm{im}(A_2^\top (AA^\top)^{-1}A_2)$ is the row span of $A_2$, the minimizer is the row span of $[0\,\, A_2]$.
    The Pl\"ucker coordinates $x^*_{j_1 \cdots j_k}$ of $[0\,\, A_2]$ are equal to $0$ if $j_1 \leq i$ and $y_{j_1 \cdots j_k}$ otherwise. 
    The result follows.

    We now turn to the case $\mathcal S_{1\cdots i,(n - (k - i)+1)\cdots n}$.
    Since the data $Y$ is generic, it has a basis of the form 
        $A = \begin{bmatrix}
            A_1 & 0\\
            A_2 & \mathrm{Id}_{k-i}
        \end{bmatrix}$
    where $A_1 \in \mathbb C^{i \times (n-k+i)}$ and $A_2 \in \mathbb C^{(k - i) \times (n - k + i)}$.
    We claim the unique critical point is obtained by setting $A_2 = 0$; this corresponds to zeroing out the appropriate Pl\"ucker coordinates of the data.
    By Lemma~\ref{lem:cut-locus-for-products}, the point $A_1 \oplus \mathrm{Id}_{k-i}$ is critical if and only if the rows of $A_1$ are in the eigenspace of $Q_{11}$ with eigenvalue $1$, where $Q_{11}$ is the upper $(n- k + i) \times (n-k+i)$ block of $Q$.
    Indeed, since the rows of $[A_1\,\, 0]$ are in the image of $Q$ by assumption, we have $[A_1\,\, 0]Q = [A_1Q_{11} \,\, A_1Q_{12}^\top] = [A_1\,\, 0]$. 
    Thus $A_1Q_{11} = A_1$.     
\end{proof}
\begin{proof}[Proof of Theorem~\ref{thm:gd-schubert}]
  The GD degree of $\mathcal S_{i+1,\cdots, i+k}$ is $1$ by Corollary~\ref{cor:small-grass}. 
    The GD degree of $\mathcal S_{1\cdots i,(n - (k - i)+1)\cdots n}$ is $1$ by Corollary~\ref{cor:Schubert-EDGD}.
    This proves the first claim. 
    The second claim is immediate from Theorem~\ref{thm:zero}
\end{proof}

\subsubsection*{Acknowledgments}
This material is based upon work supported by the National Science Foundation under Grant No. DMS-2424556 while the first author was in residence at the Institute for Computational and Experimental Research in Mathematics in Providence, RI, during the \emph{Computations on K3 Surfaces and Related Varieties} program.

In addition, the first and second authors are grateful to the Max Planck Institute for Mathematics in the Sciences in Leipzig, Germany, for hosting them in spring and summer of 2026 to work on this project with the third author.

\subsubsection*{AI Usage}
We used GPT 5.5 and Gemini 3.1 Pro, along with GPT-5.6 Sol, Claude Opus 4.8, and GPT-6 Pro accessed via \verb+ScienceBench.ai+ \cite{sciencebench}, to explore potential proof structures and pathways for the statements.
We used LLMs to refine Theorems~\ref{thm:rat-norm}, \ref{thm:defective}, Propositions~\ref{prop:gen-eigen}, \ref{prop:singular-rank-2}, and Lemma \ref{lem:n>r+s}.
In the proofs of Theorem~\ref{thm:table}, Propositions~\ref{prop:SVD}, \ref{prop:pluecker}, and Lemmas~\ref{lem:tech}, \ref{lem:envelope}, LLMs led us to discover results in the literature of which we were not aware.
Finally, the core ideas in the proofs of Theorem~\ref{thm:gd-finite} and Proposition~\ref{prop:matroid-degree}, as well as Example~\ref{ex:nonirred}, came from LLMs.
The ultimate logical derivations, formal verifications, and mathematical text were entirely developed and written by the authors.

\bibliographystyle{plain}
\bibliography{references}

@article {DFRS,
    AUTHOR = {Devriendt, Karel and Friedman, Hannah and Reinke, Bernhard and
              Sturmfels, Bernd},
     TITLE = {The two lives of the {G}rassmannian},
   JOURNAL = {Acta Univ. Sapientiae Math.},
  FJOURNAL = {Acta Universitatis Sapientiae. Mathematica. An International
              Scientific Journal of Sapientia University},
    VOLUME = {17},
      YEAR = {2025},
    NUMBER = {1},
     PAGES = {Paper No. 8, 18},
}

@book{gantmacher,
author = "Felix R. Gantmacher",
title = "The Theory of Matrices",
publisher = "Chelsea Publishing",
address = "New York",
year = "1959",
volumes = "1, 2"
}

@misc{oscar,
  key          = {OSCAR},
  organization = {The OSCAR Team},
  title        = {{OSCAR -- Open Source Computer Algebra Research system,
                  Version 1.3.1}},
  year         = {2025},
  url          = {https://www.oscar-system.org},
  }

@misc{FRS,
      title={Distance Optimization in the {G}rassmannian of Lines}, 
      author={Hannah Friedman and Andrea Rosana and Bernd Sturmfels},
      year={2026},
      note={arXiv:\verb+2601.22843+},
      eprint={2601.22843},
      archivePrefix={arXiv},
      primaryClass={math.AG}
}

@article {LLY,
    AUTHOR = {Lai, Zehua and Lim, Lek-Heng and Ye, Ke},
     TITLE = {{E}uclidean distance degree in manifold optimization},
   JOURNAL = {SIAM J. Optim.},
  FJOURNAL = {SIAM Journal on Optimization},
    VOLUME = {35},
      YEAR = {2025},
    NUMBER = {4},
     PAGES = {2402--2422},
      ISSN = {1052-6234,1095-7189},
   MRCLASS = {14M15 (15A24 58C05 58K05 90C26)},
  MRNUMBER = {4976079},
       DOI = {10.1137/25M1735032},
       URL = {https://doi.org/10.1137/25M1735032},
}

@article {PS,
    AUTHOR = {Pratt, Elizabeth and Sturmfels, Bernd},
     TITLE = {The {C}how-{L}am form},
   JOURNAL = {J. Symbolic Comput.},
  FJOURNAL = {Journal of Symbolic Computation},
    VOLUME = {131},
      YEAR = {2025},
     PAGES = {Paper No. 102450, 18},
}

@article {FH,
    AUTHOR = {Friedman, Hannah and Ho\c{s}ten, Serkan},
     TITLE = {Grassmann and flag varieties in linear algebra, optimization,
              and statistics: an algebraic perspective},
   JOURNAL = {Adv. in Appl. Math.},
  FJOURNAL = {Advances in Applied Mathematics},
    VOLUME = {180},
      YEAR = {2026},
     PAGES = {Paper No. 103130, 24},
}

@article {DHOST,
    AUTHOR = {Draisma, Jan and Horobe\c{t}, Emil and Ottaviani, Giorgio and
              Sturmfels, Bernd and Thomas, Rekha R.},
     TITLE = {The {E}uclidean distance degree of an algebraic variety},
   JOURNAL = {Found. Comput. Math.},
  FJOURNAL = {Foundations of Computational Mathematics. The Journal of the
              Society for the Foundations of Computational Mathematics},
    VOLUME = {16},
      YEAR = {2016},
    NUMBER = {1},
     PAGES = {99--149},
}

@article {dukes,
    AUTHOR = {Dukes, W. M. B.},
     TITLE = {On the number of matroids on a finite set},
   JOURNAL = {S\'em. Lothar. Combin.},
  FJOURNAL = {S\'eminaire Lotharingien de Combinatoire},
    VOLUME = {51},
      YEAR = {2004/05},
     PAGES = {Art. B51g, 12},
      ISSN = {1286-4889},
   MRCLASS = {05B35},
  MRNUMBER = {2119297},
MRREVIEWER = {Anna\ de Mier},
}

@article {BZA,
    AUTHOR = {Bendokat, Thomas and Zimmermann, Ralf and Absil, P.-A.},
     TITLE = {A {G}rassmann manifold handbook: basic geometry and
              computational aspects},
   JOURNAL = {Adv. Comput. Math.},
  FJOURNAL = {Advances in Computational Mathematics},
    VOLUME = {50},
      YEAR = {2024},
    NUMBER = {1},
     PAGES = {Paper No. 6, 51},
}

@book{vakil, 
author = {Vakil, Ravi}, 
title = {The Rising Sea: Foundations of Algebraic Geometry}, 
publisher = {Princeton University Press}, 
year = {2025}
}

@article {PW94,
    AUTHOR = {Paige, C. C. and Wei, M.},
     TITLE = {History and generality of the {${\rm CS}$} decomposition},
   JOURNAL = {Linear Algebra Appl.},
  FJOURNAL = {Linear Algebra and its Applications},
    VOLUME = {208/209},
      YEAR = {1994},
     PAGES = {303--326},
}

@misc{LR,
      title={The {G}rassmann distance complexity}, 
      author={Antonio Lerario and Andrea Rosana},
      year={2024},
      eprint={2411.16589},
      note={arXiv:\verb+2411.16589+},
      archivePrefix={arXiv},
      primaryClass={math.DG},
      url={https://arxiv.org/abs/2411.16589},
}

@book {LB,
    AUTHOR = {Lakshmibai, V. and Brown, Justin},
     TITLE = {The {G}rassmannian {V}ariety},
    SERIES = {Developments in Mathematics},
    VOLUME = {42},
 PUBLISHER = {Springer, New York},
      YEAR = {2015},
     PAGES = {x+172},
      ISBN = {978-1-4939-3081-4; 978-1-4939-3082-1},
   MRCLASS = {14M15 (13P10 14M12)},
  MRNUMBER = {3408060},
MRREVIEWER = {Li\ Li},
       DOI = {10.1007/978-1-4939-3082-1},
       URL = {https://doi.org/10.1007/978-1-4939-3082-1},
}

@book {Oxley,
    AUTHOR = {Oxley, James G.},
     TITLE = {Matroid {T}heory},
    SERIES = {Oxford Science Publications},
 PUBLISHER = {The Clarendon Press, Oxford University Press, New York},
      YEAR = {1992},
}

@misc{CL,
      title={Singular matroid realization spaces}, 
      author={Daniel Corel and Dante Luber},
      year={2023},
      eprint={2307.11915},
      archivePrefix={arXiv},
      note={arXiv:\verb+2307.11915+},
      primaryClass={math.AG},
      url={https://arxiv.org/abs/2307.11915},
}

@misc 
{sciencebench,
author       = {Christian Stump},
title         = {ScienceBench},
howpublished  = {\url{http://math.sciencebench.ai}},
url           = {http://math.ScienceBench.ai},
note          = {Accessed: \today},
}

@book {3264,
    AUTHOR = {Eisenbud, David and Harris, Joe},
     TITLE = {3264 and {A}ll {T}hat---a {S}econd {C}ourse in {A}lgebraic {G}eometry},
 PUBLISHER = {Cambridge University Press, Cambridge},
      YEAR = {2016},
     PAGES = {xiv+616},
      ISBN = {978-1-107-60272-4; 978-1-107-01708-5},
   MRCLASS = {14-01 (14C15 14M15 14N10)},
  MRNUMBER = {3617981},
MRREVIEWER = {Arnaud\ Beauville},
       DOI = {10.1017/CBO9781139062046},
       URL = {https://doi.org/10.1017/CBO9781139062046},
}

@book {MS74,
    AUTHOR = {Milnor, John W. and Stasheff, James D.},
     TITLE = {Characteristic {C}lasses},
    SERIES = {Annals of Mathematics Studies},
    VOLUME = {76},
 PUBLISHER = {Princeton University Press, Princeton, NJ; University of Tokyo
              Press, Tokyo},
      YEAR = {1974},
     PAGES = {vii+331},
   MRCLASS = {57-01 (55-02 55F40 57D20)},
  MRNUMBER = {440554},
MRREVIEWER = {F.\ Hirzebruch},
}

@misc{postnikov, 
title = {Total positivity, {G}rassmannians, and networks},
author = {Alexander Postnikov}, 
year={2006},
      eprint={0609764},
      archivePrefix={arXiv},
      note={arXiv:\verb+math/0609764+},
      primaryClass={math.CO},
      url={https://arxiv.org/abs/math/0609764},
}

@article {Ily92,
    AUTHOR = {Ilyushechkin, N. V.},
     TITLE = {The discriminant of the characteristic polynomial of a normal
              matrix},
   JOURNAL = {Mat. Zametki},
  FJOURNAL = {Matematicheskie Zametki},
    VOLUME = {51},
      YEAR = {1992},
    NUMBER = {3},
     PAGES = {16--23, 143},
}

@book {Stu02,
    AUTHOR = {Sturmfels, Bernd},
     TITLE = {Solving {S}ystems of {P}olynomial {E}quations},
    SERIES = {CBMS Regional Conference Series in Mathematics},
    VOLUME = {97},
 PUBLISHER = {Conference Board of the Mathematical Sciences, Washington, DC;
              by the American Mathematical Society, Providence, RI},
      YEAR = {2002},
     PAGES = {viii+152},
}

\bigskip

\noindent
\footnotesize \textbf{Authors' addresses:}

\medskip 
\noindent Hannah Friedman, ICERM, USA  \hfill {\tt hannahfriedman@berkeley.edu}
\\
\noindent Serkan Ho\c{s}ten,
San Francisco State University, USA
\hfill {\tt serkan@sfsu.edu}
\\
\noindent Andrea Rosana, MPI MiS Leipzig, Germany \hfill  {\tt andrea.rosana@mis.mpg.de}
\end{document}